\documentclass[11pt]{article}

\usepackage{authblk}

\usepackage[utf8]{inputenc}
\usepackage[english]{babel}
\usepackage{csquotes}

\usepackage{amsfonts, amsmath, amssymb, amsthm}
\usepackage{mathtools}
\usepackage{mathrsfs}
\usepackage{bbm}
\usepackage{breqn}
\usepackage{bm}
\usepackage{cancel}

\usepackage{fullpage}

\usepackage{ragged2e}
\usepackage{enumerate}
\usepackage{enumitem}
\usepackage{multicol}
\usepackage{tasks}
\usepackage{graphicx}
\usepackage{float}
\usepackage{wrapfig}
\usepackage{animate}
\usepackage{caption}
\usepackage{subcaption}
\usepackage{mwe}

\usepackage{tikz}
\usetikzlibrary{arrows.meta}
\usetikzlibrary{arrows}
\usepackage{tikz-cd}
\usepackage{pgfplots}
\pgfplotsset{compat=1.15}

\usetikzlibrary{shapes.geometric, positioning, calc, automata}

\definecolor{white}{rgb}{1,1,1}
\definecolor{xdxdff}{rgb}{0.49019607843137253,0.49019607843137253,1}
\definecolor{ududff}{rgb}{0.30196078431372547,0.30196078431372547,1}
\definecolor{qqffff}{rgb}{0,1,1}
\definecolor{qqffqq}{rgb}{0,1,0}
\definecolor{ffqqqq}{rgb}{1,0,0}
\definecolor{qqqqff}{rgb}{0,0,1}
\definecolor{qqqqcc}{rgb}{0,0,0.8}
\definecolor{ccqqqq}{rgb}{0.8,0,0}
\definecolor{grey}{rgb}{0.4,0.4,0.4}
\definecolor{zzffzz}{rgb}{0.85,1,0.85}
\definecolor{zzffff}{rgb}{0.85,1,1}
\usepackage{aliascnt}
\usepackage{hyperref}
\usepackage{cleveref}

\usepackage[backend=biber, url=false, isbn=false, eprint=false, maxbibnames=99, doi=false,  maxcitenames=99]{biblatex}
\newtheorem{lemma}{Lemma}[section]

\newaliascnt{theorem}{lemma}  % Create counter 'theorem' that mirrors 'lemma'
\newtheorem{theorem}[theorem]{Theorem} % Define environment using the alias
\aliascntresetthe{theorem}    % Reset the alias so it formats correctly
\Crefname{theorem}{Theorem}{Theorems} % Tell cleveref how to name it (optional but safe)

\theoremstyle{definition}
\newtheorem{definition}{Definition}[section]

\theoremstyle{plain}

\newtheorem{obs}{Observation}[section]

\newtheorem{ques}{Question}[section]

\theoremstyle{remark}
\newtheorem*{remark}{Remark}

\newcommand{\kcut}{\Delta_{k}}

\newcommand{\lno}{\left(}
\newcommand{\rno}{\right)}
\newcommand{\lsq}{\left[}
\newcommand{\rsq}{\right]}
\newcommand{\lcu}{\left\{}
\newcommand{\rcu}{\right\}}

\newcommand{\precc}{\mathrel{\prec\!\!\prec}}
\newcommand{\succc}{\mathrel{\succ\!\!\succ}}

\newcommand{\medsize}{\fontsize{11pt}{13pt}\selectfont}

\newcommand{\caseheading}[1]{%
  \par\medskip\noindent\textbf{#1}\quad\ignorespaces
}

\usepackage{lipsum}
\usepackage{blindtext}

\let\svthefootnote\thefootnote
\newcommand\freefootnote[1]{%
  \let\thefootnote\relax%
  \footnotetext{#1}%
  \let\thefootnote\svthefootnote%
}

\makeatother

\title{\textbf{On Shellability of 3-Cut Complexes of Hexagonal Grid Graphs}}
\author[1]{Himanshu Chandrakar}

\affil[1]{Department of Mathematics, Indian Institute of Technology Bhilai, \texttt{himanshuc@iitbhilai.ac.in}} 

\date{}

\begin{document}

\maketitle 

\freefootnote{2020 \textit{Mathematics Subject Classification.} 05C69, 05E45, 55P15, 57M15}
\freefootnote{\textit{Keywords and phrases.} $k$-cut complex; shellable simplicial complex; hexagonal grid graph; homotopy type}
%%% ABSTRACT %%%
\begin{abstract}
    The $k$-cut complex was recently introduced by Bayer et al. as a generalization of earlier work of Fr{\"o}berg (1990) and Eagon and Reiner (1998), and was shown to be shellable for several classes of graphs. In this article, we prove that the $3$-cut complex of the hexagonal grid graph $H_{1 \times m \times n}$ is shellable for all $m,n \geq 1$, by constructing an explicit shelling order using reverse lexicographic ordering. From this shelling, we determine the number of spanning facets, denoted by $\psi_{m,n}$, and deduce that the complex is homotopy equivalent to a wedge of $\psi_{m,n}$ spheres of dimension $\lno 2m + 2n + 2mn - 4 \rno$, where $$\psi_{m,n} = \binom{2m+2n+2mn-1}{2} - \lsq \lno 6m+2 \rno n + (2m-4) \rsq.$$

    While these topological properties can be obtained from general results of Bayer et al., we provide an explicit combinatorial construction of a shelling order, yielding a direct counting formula for the number of spheres in the wedge sum decomposition.
\end{abstract}

%%% Section 1: Introduction %%%
\section{Introduction}

Simplicial complexes arising from graph-based properties are well-studied objects in topological combinatorics. The books by Jonsson~\cite{jonsson2008simplicial} and Kozlov~\cite{kozlov2008combinatorial} are excellent works in this regard. Simplicial complexes naturally occur at the intersection of various areas of mathematics. Among these, the interaction with commutative algebra is especially prominent, where these complexes appear via Stanley–Reisner theory (see, for instance, \cite{Dochter09, dochter12, hibi92, hoch16, Stan96, van13}).

For instance, this connection has been explored in detail in recent work of Bayer et al. (see \cite{bayer2024cutB, bayer2024cut, bayer2024total}). In 1990, Fr{\"o}berg proved his celebrated theorem~\cite[Theorem~1]{froberg1990} concerning ideals generated by quadratic monomials. Subsequently, Eagon and Reiner~\cite[Proposition~8]{ER98} generalized Fr{\"o}berg’s result using the Alexander dual of the clique complex $\Delta(G)$, denoted by $\Delta_2(G)$, whose facets correspond to complements of independent sets of size~$2$. Motivated by this line of work, Bayer et al. recently introduced two classes of graph complexes: the \textit{total $k$-cut complexes} and the \textit{$k$-cut complexes}. In this article, we focus on $k$-cut complexes (see \Cref{def:cutcomplex}).

In their work, Bayer et al. established the shellability of $k$-cut complexes for a wide range of graph families. This includes trees, threshold graphs, cycles (for $k \geq 3$), complete multipartite graphs, squared paths, and grid graphs (specifically for the case $k = 3$). Additionally, they conjectured that this shellability extends to grid graphs for general $k \geq 1$, as well as to squared cycles when the number of vertices is sufficiently large relative to $k \geq 2$. More recently, Chouhan et al. \cite{Chauhan2025} studied the 3-cut complexes of squared cycle graphs and analyzed their shellability. Also, Chandrakar et al. in~\cite{chandrakar2024topologytotalcutcut}, showed that the $k$-cut complex of $\lno 2 \times n \rno$-grid graphs are shellable when $n \geq 3$ and $3 \leq k \leq 2n-2$.

In this article, we study the shellability of the $3$-cut complex of hexagonal grid graphs (see \Cref{fig: graph of H(1_4_6)}) and determine its homotopy type. We note that these results can be derived from \cite[Proposition~5.4]{bayer2024cut} and \cite[Proposition~5.5]{bayer2024cutB}, where shellability is established using the notion of \textit{minimal forbidden subgraphs}, and from \cite[Proposition~2.18]{bayer2024cutB}, where the homotopy type is computed via the \textit{M{\"o}bius function of the face lattice} of the $3$-cut complex of a graph.

While the general results of Bayer et al. yield both shellability and the homotopy type of this complex through the use of additional structural notions, we provide an independent proof using only the definition of shellability. Our approach is based solely on constructing an explicit shelling order and describing the associated spanning facets. It is worth noting that determining whether a simplicial complex is shellable is an \texttt{NP}-complete problem \cite{shellabilityNPcomplete}. Consequently, finding an explicit shelling order, even for a complex already known to be shellable, is computationally difficult and often requires a detailed understanding of its combinatorial structure.

In particular, we prove the following results,

\begin{theorem}[{\Cref{thm: shellability of hexagonal grid graphs}}]
    For $m, n\geq 1$, the simplicial complex $\Delta_3\left(H_{1\times m \times n}\right)$ is shellable.
\end{theorem}

The existence of a shelling order allows us to determine the homotopy type of the complex. This leads to the following theorem.

\begin{theorem}[{\Cref{thm: spanning facets and homotopy type}}]
    The 3-cut complex of $H_{1 \times m \times n}$ has the homotopy type of wedge of $\psi_{m,n}$ many $\lno 2m+2n+2mn-4 \rno$-dimensional spheres, where, $$\psi_{m,n} = \binom{2m+2n+2mn-1}{2} - \lsq \lno 6m+2 \rno n + (2m-4) \rsq.$$
    In other words,
    $$
    \Delta_3 \lno H_{1 \times m \times n} \rno \simeq \bigvee\limits_{\psi_{m,n}}{\mathbb{S}^{\lno 2m+2n+2mn-4 \rno}}.
    $$
\end{theorem}

\caseheading{Organization of the article:}

In \Cref{section: 2 (Preliminaries)}, we define the hexagonal grid graphs $H_{1\times m\times n}$ for $m,n \geq 1$ and recall the necessary background on simplicial complexes, including the notions of shellability and spanning facets that will be used throughout the paper. In \Cref{section: 3 (Shellability of 3-cut complex of hexagonal grid graph)}, we introduce an explicit shelling order on the facets of the $3$-cut complex of $H_{1\times m\times n}$, defined via reverse lexicographic ordering, and prove that it indeed yields a shelling. Finally, in \Cref{section: 4 (Spanning facets and the homotopy type of the 3-cut complex of hexagonal tiling)}, we characterize all spanning facets associated with this shelling order and use this description to determine the homotopy type of $\Delta_3\lno H_{1\times m\times n} \rno$.

%%% Section 2: Preliminaries %%%
\section{Preliminaries}\label{section: 2 (Preliminaries)}

In this section, we introduce the key concepts and results necessary for discussing the article. For $k_1, k_2 \in \mathbb{N}$, such that $1 \leq k_1 < k_2$, we denote
\begin{align*}
    [k_1] &= \lcu 1, 2, 3, \dots, k_1 \rcu; \\
    [k_1, k_2] &= \lcu k_1, k_1+1, k_1+2, \dots, k_2 \rcu.
\end{align*}

\subsection{Basics of graph theory}

A \textit{\textbf{graph}} $G$ is an ordered pair $\lno V(G), E(G) \rno$, where $V(G)$ is the set of vertices and $$E(G) \subseteq \lcu \{u,v\} \mid u,v \in V(G),\ u \neq v \rcu$$
is the set of edges. We use the notation $uv$ to denote edge $\{u,v\}$. We define the \textbf{open neighborhood} of the vertex $v$ in $G$, denoted by $N_G(v)$ as, $$N_G \lno v \rno = \lcu u \in V(G)\ \middle|\ u\neq v\ \text{and}\ \left\{ u,v \right\} \in E \lno G \rno \rcu.$$ Similarly, the \textbf{closed neighborhood} of the vertex $v$ in $G$, denoted by $N_G[v]$ is defined as, $$N_G[v] = N_G(v) \sqcup \lcu v \rcu.$$ When the underlying graph is clear from the context, we omit the subscript and write $N(v)$ and $N[v]$ to denote the open and closed neighborhoods of $v$, respectively.

Let $S$ be a subset of $V(G)$. The \textit{\textbf{induced subgraph}} of $G$ on $S$, denoted by $G\lsq S \rsq$, is the graph whose vertex set is $S$ and whose edge set consists of all edges of $G$ that have both vertices in $S$. We denote the induced
subgraph $G\lsq V\lno G \rno \setminus H \rsq$ by $G \setminus H$, where $H \subset V\lno G \rno$. We now proceed to formally define the hexagonal grid graphs. 

\subsection{Labelling of vertices in the general hexagonal grid graph.}

The hexagonal grid graphs are generally defined using 3 variables $r,s,t \geq 1$ and are denoted by $H_{r \times s \times t}$. For our discussion, we fix $r=1$, that is, we will deal with $H_{1 \times s \times t}$. Here, our focus is primarily on the labelling of vertices, as it will play a crucial role in defining the shelling order. 

\begin{definition}[Hexagonal grid graphs]
    For $m,n \geq 1$, we  define the hexagonal grid graph, $H_{1 \times m \times n}$, as a graph with the following vertices and edges,
\begin{align*}
    V\left(H_{1 \times m \times n}\right) &= \left[2m + 2n + 2mn\right];\\
    E\left(H_{1 \times m \times n}\right) &= \mathcal{E}_1\ \sqcup\ \mathcal{E}_2\ \sqcup\ \mathcal{E}_3\ \sqcup\ \mathcal{E}_4.
\end{align*}
where, for $1 \leq i \leq 4$, 
\begin{align*}
    \mathcal{E}_1\ &=\ \lcu (i)\lno i+m+n+mn \rno\ \middle|\ i \in \lsq m \rsq \cup \lsq 1+n+nm,\ m+n+nm \rsq \rcu;\\
    %\mathcal{E}_2\ &=\ \lcu (i)\lno i+m+n+mn+1 \rno\ \middle|\ i \in \lsq m+1,\ m+n+nm \rsq \rcu\\
    \mathcal{E}_2\ &=\ \lcu (i)\lno i+n+nm \rno\ \middle|\ i \in \lsq m+1,\ m+n+nm \rsq \rcu;\\
    \mathcal{E}_3\ &=\ \lcu (i)\lno i+m+n+mn+1 \rno\ \middle|\ i \in \lsq m,  n+nm,\ m+n+nm-1 \rsq \rcu;\\
    \mathcal{E}_4\ &=\ \lcu (i)\lno i+m+n+mn+2 \rno\ \middle|\ i \in \lsq m+1,\ n+nm-2 \rsq \rcu.
\end{align*}
\end{definition}

We will use $\left[2m + 2n + 2mn\right]$ and $V\left(H_{1 \times m \times n}\right)$ interchangeably. For example, if $m = 4$ and $n = 6$, then the graph $H_{1 \times 4 \times 6}$ will look like \Cref{fig: graph of H(1_4_6)}. Since the argument in the main result of the proof is rather intricate, we will refer to this example frequently for better understanding.

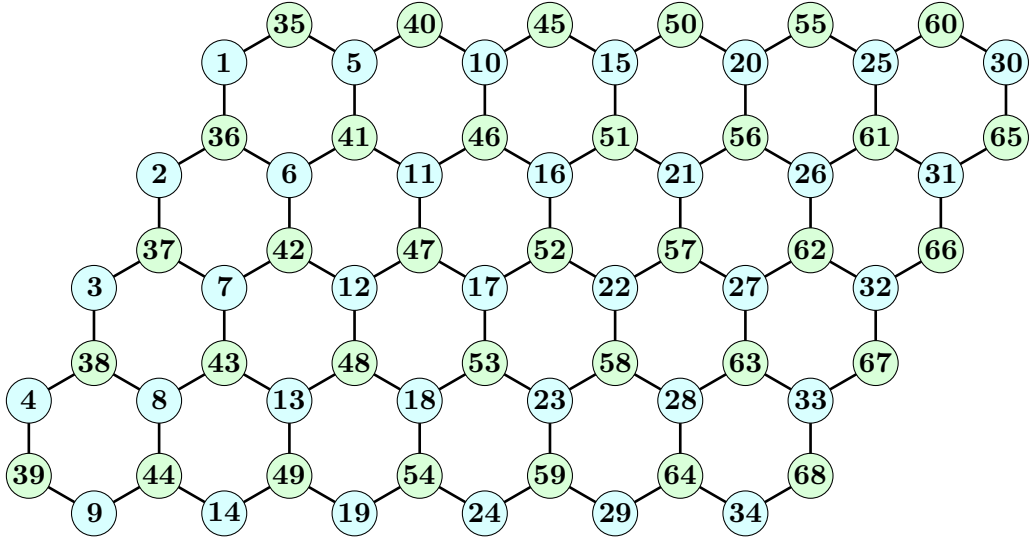
\begin{figure}[H]
\centering
    \begin{tikzpicture}[line cap=round,line join=round,>=triangle 45,x=1cm,y=1cm]
        \clip(-0.5,-0.5) rectangle (14,7.5);
        \draw [line width=1pt] (0.4330127018922193,1.75)-- (0.4330127018922193,0.75);
        \draw [line width=1pt] (0.4330127018922193,0.75)-- (1.299038105676658,0.25);
        \draw [line width=1pt] (1.299038105676658,0.25)-- (2.1650635094610964,0.75);
        \draw [line width=1pt] (2.1650635094610964,0.75)-- (2.1650635094610964,1.75);
        \draw [line width=1pt] (2.1650635094610964,1.75)-- (1.299038105676658,2.25);
        \draw [line width=1pt] (1.299038105676658,2.25)-- (0.4330127018922193,1.75);
        \draw [line width=1pt] (2.1650635094610964,1.75)-- (3.031088913245535,2.25);
        \draw [line width=1pt] (3.031088913245535,2.25)-- (3.8971143170299736,1.75);
        \draw [line width=1pt] (3.8971143170299736,1.75)-- (3.8971143170299736,0.75);
        \draw [line width=1pt] (3.8971143170299736,0.75)-- (3.031088913245535,0.25);
        \draw [line width=1pt] (3.031088913245535,0.25)-- (2.1650635094610964,0.75);
        \draw [line width=1pt] (3.8971143170299736,1.75)-- (4.763139720814412,2.25);
        \draw [line width=1pt] (4.763139720814412,2.25)-- (5.629165124598851,1.75);
        \draw [line width=1pt] (5.629165124598851,1.75)-- (5.629165124598851,0.75);
        \draw [line width=1pt] (5.629165124598851,0.75)-- (4.763139720814412,0.25);
        \draw [line width=1pt] (4.763139720814412,0.25)-- (3.8971143170299736,0.75);
        \draw [line width=1pt] (5.629165124598851,1.75)-- (6.495190528383289,2.25);
        \draw [line width=1pt] (6.495190528383289,2.25)-- (7.361215932167728,1.75);
        \draw [line width=1pt] (7.361215932167728,1.75)-- (7.361215932167728,0.75);
        \draw [line width=1pt] (7.361215932167728,0.75)-- (6.495190528383289,0.25);
        \draw [line width=1pt] (6.495190528383289,0.25)-- (5.629165124598851,0.75);
        \draw [line width=1pt] (7.361215932167728,1.75)-- (8.227241335952167,2.25);
        \draw [line width=1pt] (8.227241335952167,2.25)-- (9.093266739736606,1.75);
        \draw [line width=1pt] (9.093266739736606,1.75)-- (9.093266739736606,0.75);
        \draw [line width=1pt] (9.093266739736606,0.75)-- (8.227241335952167,0.25);
        \draw [line width=1pt] (8.227241335952167,0.25)-- (7.361215932167728,0.75);
        \draw [line width=1pt] (9.093266739736606,1.75)-- (9.959292143521044,2.25);
        \draw [line width=1pt] (9.959292143521044,2.25)-- (10.825317547305483,1.75);
        \draw [line width=1pt] (10.825317547305483,1.75)-- (10.825317547305483,0.75);
        \draw [line width=1pt] (10.825317547305483,0.75)-- (9.959292143521044,0.25);
        \draw [line width=1pt] (9.959292143521044,0.25)-- (9.093266739736606,0.75);
        \draw [line width=1pt] (1.299038105676658,2.25)-- (1.299038105676658,3.25);
        \draw [line width=1pt] (1.299038105676658,3.25)-- (2.1650635094610964,3.75);
        \draw [line width=1pt] (2.1650635094610964,3.75)-- (3.031088913245535,3.25);
        \draw [line width=1pt] (3.031088913245535,3.25)-- (3.031088913245535,2.25);
        \draw [line width=1pt] (3.031088913245535,3.25)-- (3.8971143170299736,3.75);
        \draw [line width=1pt] (3.8971143170299736,3.75)-- (4.763139720814412,3.25);
        \draw [line width=1pt] (4.763139720814412,3.25)-- (4.763139720814412,2.25);
        \draw [line width=1pt] (4.763139720814412,3.25)-- (5.629165124598851,3.75);
        \draw [line width=1pt] (5.629165124598851,3.75)-- (6.495190528383289,3.25);
        \draw [line width=1pt] (6.495190528383289,3.25)-- (6.495190528383289,2.25);
        \draw [line width=1pt] (6.495190528383289,3.25)-- (7.361215932167728,3.75);
        \draw [line width=1pt] (7.361215932167728,3.75)-- (8.227241335952167,3.25);
        \draw [line width=1pt] (8.227241335952167,3.25)-- (8.227241335952167,2.25);
        \draw [line width=1pt] (8.227241335952167,3.25)-- (9.093266739736606,3.75);
        \draw [line width=1pt] (9.093266739736606,3.75)-- (9.959292143521044,3.25);
        \draw [line width=1pt] (9.959292143521044,3.25)-- (9.959292143521044,2.25);
        \draw [line width=1pt] (9.959292143521044,3.25)-- (10.825317547305483,3.75);
        \draw [line width=1pt] (10.825317547305483,3.75)-- (11.69134295108992,3.25);
        \draw [line width=1pt] (11.69134295108992,3.25)-- (11.69134295108992,2.25);
        \draw [line width=1pt] (11.69134295108992,2.25)-- (10.825317547305483,1.75);
        \draw [line width=1pt] (2.1650635094610964,3.75)-- (2.1650635094610964,4.75);
        \draw [line width=1pt] (2.1650635094610964,4.75)-- (3.031088913245535,5.25);
        \draw [line width=1pt] (3.031088913245535,5.25)-- (3.8971143170299736,4.75);
        \draw [line width=1pt] (3.8971143170299736,4.75)-- (3.8971143170299736,3.75);
        \draw [line width=1pt] (3.8971143170299736,4.75)-- (4.763139720814412,5.25);
        \draw [line width=1pt] (4.763139720814412,5.25)-- (5.629165124598851,4.75);
        \draw [line width=1pt] (5.629165124598851,4.75)-- (5.629165124598851,3.75);
        \draw [line width=1pt] (5.629165124598851,4.75)-- (6.495190528383289,5.25);
        \draw [line width=1pt] (6.495190528383289,5.25)-- (7.361215932167728,4.75);
        \draw [line width=1pt] (7.361215932167728,4.75)-- (7.361215932167728,3.75);
        \draw [line width=1pt] (7.361215932167728,4.75)-- (8.227241335952167,5.25);
        \draw [line width=1pt] (8.227241335952167,5.25)-- (9.093266739736606,4.75);
        \draw [line width=1pt] (9.093266739736606,4.75)-- (9.093266739736606,3.75);
        \draw [line width=1pt] (9.093266739736606,4.75)-- (9.959292143521044,5.25);
        \draw [line width=1pt] (9.959292143521044,5.25)-- (10.825317547305483,4.75);
        \draw [line width=1pt] (10.825317547305483,4.75)-- (10.825317547305483,3.75);
        \draw [line width=1pt] (10.825317547305483,4.75)-- (11.69134295108992,5.25);
        \draw [line width=1pt] (11.69134295108992,5.25)-- (12.55736835487436,4.75);
        \draw [line width=1pt] (12.55736835487436,4.75)-- (12.55736835487436,3.75);
        \draw [line width=1pt] (12.55736835487436,3.75)-- (11.69134295108992,3.25);
        \draw [line width=1pt] (3.031088913245535,5.25)-- (3.031088913245535,6.25);
        \draw [line width=1pt] (3.031088913245535,6.25)-- (3.8971143170299736,6.75);
        \draw [line width=1pt] (3.8971143170299736,6.75)-- (4.763139720814412,6.25);
        \draw [line width=1pt] (4.763139720814412,6.25)-- (4.763139720814412,5.25);
        \draw [line width=1pt] (4.763139720814412,6.25)-- (5.629165124598851,6.75);
        \draw [line width=1pt] (5.629165124598851,6.75)-- (6.495190528383289,6.25);
        \draw [line width=1pt] (6.495190528383289,6.25)-- (6.495190528383289,5.25);
        \draw [line width=1pt] (6.495190528383289,6.25)-- (7.361215932167728,6.75);
        \draw [line width=1pt] (7.361215932167728,6.75)-- (8.227241335952167,6.25);
        \draw [line width=1pt] (8.227241335952167,6.25)-- (8.227241335952167,5.25);
        \draw [line width=1pt] (8.227241335952167,6.25)-- (9.093266739736606,6.75);
        \draw [line width=1pt] (9.093266739736606,6.75)-- (9.959292143521044,6.25);
        \draw [line width=1pt] (9.959292143521044,6.25)-- (9.959292143521044,5.25);
        \draw [line width=1pt] (9.959292143521044,6.25)-- (10.825317547305483,6.75);
        \draw [line width=1pt] (10.825317547305483,6.75)-- (11.69134295108992,6.25);
        \draw [line width=1pt] (11.69134295108992,6.25)-- (11.69134295108992,5.25);
        \draw [line width=1pt] (11.69134295108992,6.25)-- (12.55736835487436,6.75);
        \draw [line width=1pt] (12.55736835487436,6.75)-- (13.423393758658799,6.25);
        \draw [line width=1pt] (13.423393758658799,6.25)-- (13.423393758658799,5.25);
        \draw [line width=1pt] (13.423393758658799,5.25)-- (12.55736835487436,4.75);
        %\draw (2.9613654461604324,6.3870375594109365) node[anchor=north west] {\textbf{1}};
        
       \begin{scriptsize}
            %1
            \draw [fill=zzffff] (3.031088913245535,6.25) circle (8.5pt);
            \node at (3.025,6.25) {\medsize \textbf{1}};
            %2
            \draw [fill=zzffff] (2.1650635094610964,4.75) circle (8.5pt);
            \node at (2.17,4.77) {\medsize \textbf{2}};
            %3
            \draw [fill=zzffff] (1.299038105676658,3.25) circle (8.5pt);
            \node at (1.31,3.26) {\medsize \textbf{3}};
            %4
            \draw [fill=zzffff] (0.4330127018922193,1.75) circle (8.5pt);
            \node at (0.43,1.76) {\medsize \textbf{4}};
            %5
            \draw [fill=zzffff] (4.763139720814412,6.25) circle (8.5pt);
            \node at (4.76,6.25) {\medsize \textbf{5}};
            %6
            \draw [fill=zzffff] (3.8971143170299736,4.75) circle (8.5pt);
            \node at (3.90,4.77) {\medsize \textbf{6}};
            %7
            \draw [fill=zzffff] (3.031088913245535,3.25) circle (8.5pt);
            \node at (3.04,3.26) {\medsize \textbf{7}};
            %8
            \draw [fill=zzffff] (2.1650635094610964,1.75) circle (8.5pt);
            \node at (2.17,1.76) {\medsize \textbf{8}};
            %9
            \draw [fill=zzffff] (1.299038105676658,0.25) circle (8.5pt);
            \node at (1.31,0.26) {\medsize \textbf{9}};
            %10
            \draw [fill=zzffff] (6.495190528383289,6.25) circle (8.5pt);
            \node at (6.495,6.25) {\medsize \textbf{10}};
            %11
            \draw [fill=zzffff] (5.629165124598851,4.75) circle (8.5pt);
            \node at (5.63,4.77) {\medsize \textbf{11}};
            %12
            \draw [fill=zzffff] (4.763139720814412,3.25) circle (8.5pt);
            \node at (4.76,3.26) {\medsize \textbf{12}};
            %13
            \draw [fill=zzffff] (3.8971143170299736,1.75) circle (8.5pt);
            \node at (3.89,1.76) {\medsize \textbf{13}};
            %14
            \draw [fill=zzffff] (3.031088913245535,0.25) circle (8.5pt);
            \node at (3.03,0.27) {\medsize \textbf{14}};
            %15
            \draw [fill=zzffff] (8.227241335952167,6.25) circle (8.5pt);
            \node at (8.23,6.25) {\medsize \textbf{15}};
            %16
            \draw [fill=zzffff] (7.361215932167728,4.75) circle (8.5pt);
            \node at (7.36,4.77) {\medsize \textbf{16}};
            %17
            \draw [fill=zzffff] (6.495190528383289,3.25) circle (8.5pt);
            \node at (6.48,3.26) {\medsize \textbf{17}};
            %18
            \draw [fill=zzffff] (5.629165124598851,1.75) circle (8.5pt);
            \node at (5.63,1.76) {\medsize \textbf{18}};
            %19
            \draw [fill=zzffff] (4.763139720814412,0.25) circle (8.5pt);
            \node at (4.76,0.26) {\medsize \textbf{19}};
            %20
            \draw [fill=zzffff] (9.959292143521044,6.25) circle (8.5pt);
            \node at (9.965,6.25) {\medsize \textbf{20}};
            %21
            \draw [fill=zzffff] (9.093266739736606,4.75) circle (8.5pt);
            \node at (9.1,4.77) {\medsize \textbf{21}};
            %22
            \draw [fill=zzffff] (8.227241335952167,3.25) circle (8.5pt);
            \node at (8.22,3.26) {\medsize \textbf{22}};
            %23
            \draw [fill=zzffff] (7.361215932167728,1.75) circle (8.5pt);
            \node at (7.36,1.76) {\medsize \textbf{23}};
            %24
            \draw [fill=zzffff] (6.495190528383289,0.25) circle (8.5pt);
            \node at (6.49,0.26) {\medsize \textbf{24}};
            %25
            \draw [fill=zzffff] (11.69134295108992,6.25) circle (8.5pt);
            \node at (11.7,6.25) {\medsize \textbf{25}};
            %26
            \draw [fill=zzffff] (10.825317547305483,4.75) circle (8.5pt);
            \node at (10.83,4.77) {\medsize \textbf{26}};
            %27
            \draw [fill=zzffff] (9.959292143521044,3.25) circle (8.5pt);
            \node at (9.95,3.26) {\medsize \textbf{27}};
            %28
            \draw [fill=zzffff] (9.093266739736606,1.75) circle (8.5pt);
            \node at (9.1,1.76) {\medsize \textbf{28}};
            %29
            \draw [fill=zzffff] (8.227241335952167,0.25) circle (8.5pt);
            \node at (8.23,0.26) {\medsize \textbf{29}};
            %30
            \draw [fill=zzffff] (13.423393758658799,6.25) circle (8.5pt);
            \node at (13.43,6.25) {\medsize \textbf{30}};
            %31
            \draw [fill=zzffff] (12.55736835487436,4.75) circle (8.5pt);
            \node at (12.56,4.77) {\medsize \textbf{31}};
            %32
            \draw [fill=zzffff] (11.69134295108992,3.25) circle (8.5pt);
            \node at (11.69,3.26) {\medsize \textbf{32}};
            %33
            \draw [fill=zzffff] (10.825317547305483,1.75) circle (8.5pt);
            \node at (10.83,1.76) {\medsize \textbf{33}};
            %34
            \draw [fill=zzffff] (9.959292143521044,0.25) circle (8.5pt);
            \node at (9.96,0.26) {\medsize \textbf{34}};
            %35
            \draw [fill=zzffzz] (3.8971143170299736,6.75) circle (8.5pt);
            \node at (3.9,6.76) {\medsize \textbf{35}};
            %36
            \draw [fill=zzffzz] (3.031088913245535,5.25) circle (8.5pt);
            \node at (3.03,5.26) {\medsize \textbf{36}};
            %37
            \draw [fill=zzffzz] (2.1650635094610964,3.75) circle (8.5pt);
            \node at (2.16,3.76) {\medsize \textbf{37}};
            %38
            \draw [fill=zzffzz] (1.299038105676658,2.25) circle (8.5pt);
            \node at (1.30,2.26) {\medsize \textbf{38}};
            %39
            \draw [fill=zzffzz] (0.4330127018922193,0.75) circle (8.5pt);
            \node at (0.432,0.76) {\medsize \textbf{39}};            
            %40
            \draw [fill=zzffzz] (5.629165124598851,6.75) circle (8.5pt);
            \node at (5.63,6.76) {\medsize \textbf{40}};
            %41
            \draw [fill=zzffzz] (4.763139720814412,5.25) circle (8.5pt);
            \node at (4.76,5.26) {\medsize \textbf{41}};
            %42
            \draw [fill=zzffzz] (3.8971143170299736,3.75) circle (8.5pt);
            \node at (3.89,3.76) {\medsize \textbf{42}};
            %43
            \draw [fill=zzffzz] (3.031088913245535,2.25) circle (8.5pt);
            \node at (3.03,2.26) {\medsize \textbf{43}};
            %44
            \draw [fill=zzffzz] (2.1650635094610964,0.75) circle (8.5pt);
            \node at (2.16,0.76) {\medsize \textbf{44}};
            %45
            \draw [fill=zzffzz] (7.361215932167728,6.75) circle (8.5pt);
            \node at (7.36,6.76) {\medsize \textbf{45}};
            %46
            \draw [fill=zzffzz] (6.495190528383289,5.25) circle (8.5pt);
            \node at (6.49,5.26) {\medsize \textbf{46}};
            %47
            \draw [fill=zzffzz] (5.629165124598851,3.75) circle (8.5pt);
            \node at (5.62,3.76) {\medsize \textbf{47}};
            %48
            \draw [fill=zzffzz] (4.763139720814412,2.25) circle (8.5pt);
            \node at (4.76,2.26) {\medsize \textbf{48}};
            %49
            \draw [fill=zzffzz] (3.8971143170299736,0.75) circle (8.5pt);
            \node at (3.89,0.76) {\medsize \textbf{49}};
            %50
            \draw [fill=zzffzz] (9.093266739736606,6.75) circle (8.5pt);
            \node at (9.1,6.76) {\medsize \textbf{50}};
            %51
            \draw [fill=zzffzz] (8.227241335952167,5.25) circle (8.5pt);
            \node at (8.23,5.26) {\medsize \textbf{51}};
            %52
            \draw [fill=zzffzz] (7.361215932167728,3.75) circle (8.5pt);
            \node at (7.36,3.76) {\medsize \textbf{52}};
            %53
            \draw [fill=zzffzz] (6.495190528383289,2.25) circle (8.5pt);
            \node at (6.49,2.26) {\medsize \textbf{53}};
            %54
            \draw [fill=zzffzz] (5.629165124598851,0.75) circle (8.5pt);
            \node at (5.62,0.76) {\medsize \textbf{54}};
            %55
            \draw [fill=zzffzz] (10.825317547305483,6.75) circle (8.5pt);
            \node at (10.83,6.76) {\medsize \textbf{55}};
            %56
            \draw [fill=zzffzz] (9.959292143521044,5.25) circle (8.5pt);
            \node at (9.96,5.26) {\medsize \textbf{56}};
            %57
            \draw [fill=zzffzz] (9.093266739736606,3.75) circle (8.5pt);
            \node at (9.09,3.76) {\medsize \textbf{57}};
            %58
            \draw [fill=zzffzz] (8.227241335952167,2.25) circle (8.5pt);
            \node at (8.23,2.26) {\medsize \textbf{58}};
            %59
            \draw [fill=zzffzz] (7.361215932167728,0.75) circle (8.5pt);
            \node at (7.36,0.76) {\medsize \textbf{59}};
            %60
            \draw [fill=zzffzz] (12.55736835487436,6.75) circle (8.5pt);
            \node at (12.56,6.76) {\medsize \textbf{60}};
            %61
            \draw [fill=zzffzz] (11.69134295108992,5.25) circle (8.5pt);
            \node at (11.69,5.26) {\medsize \textbf{61}};
            %62
            \draw [fill=zzffzz] (10.825317547305483,3.75) circle (8.5pt);
            \node at (10.82,3.76) {\medsize \textbf{62}};
            %63
            \draw [fill=zzffzz] (9.959292143521044,2.25) circle (8.5pt);
            \node at (9.96,2.26) {\medsize \textbf{63}};
            %64
            \draw [fill=zzffzz] (9.093266739736606,0.75) circle (8.5pt);
            \node at (9.09,0.76) {\medsize \textbf{64}};
            %65
            \draw [fill=zzffzz] (13.423393758658799,5.25) circle (8.5pt);
            \node at (13.43,5.26) {\medsize \textbf{65}};
            %66
            \draw [fill=zzffzz] (12.55736835487436,3.75) circle (8.5pt);
            \node at (12.56,3.76) {\medsize \textbf{66}};
            %67
            \draw [fill=zzffzz] (11.69134295108992,2.25) circle (8.5pt);
            \node at (11.69,2.26) {\medsize \textbf{67}};
            %68
            \draw [fill=zzffzz] (10.825317547305483,0.75) circle (8.5pt);
            \node at (10.82,0.76) {\medsize \textbf{68}};
        \end{scriptsize}
    \end{tikzpicture}
    \caption{The graph $H_{1 \times 4 \times 6}$.}
    \label{fig: graph of H(1_4_6)}
\end{figure}

Before proceeding further, we fix some notations and carefully list some important observations that arise directly from the construction of $H_{1 \times m \times n}$. For $m,n \geq 1$, let $V\lno H_{1 \times m \times n} \rno = V_1 \sqcup V_2$, where, 
\begin{align*}
    V_1 &= \lsq m+n+mn \rsq;\\
    V_2 &= \lsq m+n+mn+1,\ 2m+2n+2mn\rsq.
\end{align*}

For instance, in \Cref{fig: graph of H(1_4_6)}, the blue coloured vertices belong to $V_1$ and the green coloured vertices belong to $V_2$.

We say that a subset $A \subseteq V\lno H_{1 \times m \times n} \rno$ is \textit{connected} if the induced subgraph $H_{1 \times m \times n}[A]$ is connected. It is easy to see that a connected subset of size $3$ in $V\lno H_{1 \times m \times n} \rno$ is always a $P_3$ (path graph on three vertices). That is, if $\lcu a,b,c\rcu \subseteq V\lno H_{1 \times m \times n} \rno$ is connected, then $a$ and $c$ are the \textit{endpoints} (end vertices) of $P_3$, each adjacent to $b$, the \textit{mid-point} (vertex in between the endpoints). Also, we will say that $\lcu a,b,c\rcu$ induces a $P_3$ in $H_{1 \times m \times n}$.

\begin{obs}\label{Observation 1}
    Let $\lcu a,b,c \rcu$ induces a $P_3$ in $H_{1 \times m \times n}$ with endpoints $a$ and $c$, then exactly on of the following hold:
    \begin{enumerate}[label = \roman*.]
        \item $a,c \in V_1$ and $b \in V_2$ implying $a<c<b$; or,
        \item $a,c \in V_2$ and $b \in V_1$ implying $b<a<c$.
    \end{enumerate}

    This means both $a$ and $c$ will either be in $V_1$ or $V_2$, but not one in each, and as a result, $b$ will be in $V_2$ or $V_1$, respectively.
\end{obs}

\begin{obs}\label{Observation 2}
    If $\lcu a,b,c \rcu$ induces $P_3$ in $H_{1 \times m \times n}$ with endpoints $a$ and $c$, then these endpoints are unique; that is, no other induced $P_3$ in $H_{1 \times m \times n}$ has $a$ and $c$ as its endpoints.
\end{obs}

\begin{obs}\label{obs: V1 nbhd c}
    If $v \in V_1$ then $N(v) \subset V_2$. Similarly, if $v \in V_2$ then $N(v) \subset V_1$. Recall that $N(v)$ is the open neighborhood of the vertex $v$.
\end{obs}

\subsection{Topological preliminaries}

We follow Hatcher’s \textit{Algebraic Topology}~\cite{hatcher2005algebraic} and Kozlov’s \textit{Combinatorial Algebraic Topology}~\cite{kozlov2008combinatorial} as our primary references for standard definitions and terminology.

A \textit{simplicial complex} $\mathcal{K}$ is a non-empty family of subsets of a finite set $V$, referred to as the vertex set, with the property that it is closed under taking subsets, that is, for every $\tau \in \mathcal{K}$ and $\sigma \subset \tau$, we have $\sigma \in \mathcal{K}$. We call an element $\sigma \in K$ a \textit{face} of $K$. A face consisting of $k+1$ vertices is called a $k$-dimensional face, and the maximum dimension among all faces is called the dimension of the simplicial complex. The maximal faces of a simplicial complex, that is, those faces not properly contained in any other face, are called \emph{facets}. 

Since simplicial complexes are closed under taking subsets, it suffices to define their facets, because all other faces are obtained by taking subsets of these. Thus, if a simplicial complex $\mathcal{K}$ has facets $F_1, F_2, \dots, F_t$, we say that $\mathcal{K}$ is generated by these facets, and we write $$\mathcal{K} = \left \langle F_1, F_2, \dots, F_t \right \rangle.$$ 

We now recall the definition of the $k$-cut complex.

\begin{definition}[\cite{bayer2024cut}, Definition~2.7]\label{def:cutcomplex}
    Let $G = (V, E)$ be a graph and let $k \ge 1$. 
    The \textbf{\emph{$k$-cut complex}} of $G$, denoted by $\kcut \lno G \rno$, is the simplicial complex whose facets are the complements of subsets of $V(G)$ of size $k$ whose induced subgraphs are disconnected. Equivalently, a subset $\sigma \subseteq V(G)$ is a face of $\kcut\lno G\rno$ if and only if its complement $V(G) \setminus \sigma$ induces a disconnected subgraph on $k$ vertices.
    $$ \Delta_k \lno G \rno = \left \langle \sigma \in V(G)\ \middle|\ G\lsq V(G) \setminus \sigma \rsq \text{ is disconnected and } \left| V(G) \setminus \sigma \right| = k \right \rangle.$$
\end{definition}

\subsubsection{Shellable simplicial complexes and spanning facets}

A crucial combinatorial concept that reveals important topological information about a simplicial complex is shellability. We define this as follows,

\begin{definition}[\cite{kozlov2008combinatorial}, Defintion~12.1]\label{def:shellable}
    A simplicial complex $\mathcal{K}$ is said to be \textbf{\textit{shellable}} if the facets of $\mathcal{K}$ can be put together in a linear order $F_1, F_2, \dots, F_t$ such that the subcomplex $$\lno \bigcup_{s=1}^{j-1}\lno F_s \rno \rno \cap \lno F_j \rno$$ is pure and of dimension $\lno \dim \lno F_j \rno -1 \rno$ for each $2 \leq j \leq t$. Such an ordering on $\mathcal{K}$ is called a \textit{\textbf{shelling order}}.
\end{definition}   

The definition of shellability mentioned above is commonly used. We now discuss an alternative method for determining a shelling order, which will be useful in many cases. A simplicial complex \(\Delta\) has a shelling order \(F_1, F_2, \dots, F_t\) of its facets if and only if for any \(i, j\) satisfying \(1 \leq i < j \leq t\), there exists \(1 \leq r < j\) such that
\begin{align}\label{shellability condition}
 \tag{$*$}
    F_r \cap F_j = F_j \setminus \{\lambda\}, \quad \text{for some } \lambda \in F_j \setminus F_i.
\end{align}

A facet \(F_k\) \((1 < k \leq t)\) is called \emph{\textbf{spanning}} with respect to the given shelling order if the \textit{boundary} of $F_k$ is contained in union of facets coming before $F_k$ in shelling order, that is, $$\partial(F_k) \subseteq \bigcup_{i=1}^{k-1} F_i.$$

Here, $\partial(F_k)$ denotes the boundary of $F_k$. To fix notation, recall that a facet $f$
of a simplicial complex is a $d$-simplex, where $|f| = d+1$. We now define the boundary of a simplex.

\begin{definition}[Boundary of a simplex]
    Let $\sigma = \{v_0, v_1, \dots, v_n\}$ be an $n$-simplex. The \textit{boundary} of $\sigma$, denoted by $\partial \sigma$, is the simplicial complex consisting of all proper faces of $\sigma$, that is, $$\partial \sigma = \{ \tau \subsetneq \sigma \}.$$
\end{definition}

It is easy to check that \(F_k\) is a spanning facet if for each \(\lambda \in F_k\), there exists \(r < k\) such that
\begin{align}\label{spanning facets condition}
    F_r \cap F_k = F_k \setminus \{\lambda\}. \tag{$\#$}
\end{align}

\begin{theorem}[{\cite[Theorem 12.3]{kozlov2008combinatorial}}]\label{lemma: homotopy type using shellability}
Let \(\Delta\) be a pure shellable simplicial complex of dimension \(d\). Then \(\Delta\) has the homotopy type of a wedge of \(\beta\) spheres of dimension \(d\), where \(\beta\) is the number of spanning facets in a given shelling order. Hence,
\[
\Delta \simeq \bigvee_{\beta} S^d.
\]
\end{theorem}

%%% Sections 3: Shellability of 3-Cut Complex of Hexagonal Grid Graphs %%%
\section{Shellability of 3-cut complex of hexagonal grid graph}\label{section: 3 (Shellability of 3-cut complex of hexagonal grid graph)}

In this section, we prove that the 3-cut complex of a hexagonal tiling $H_{1\times  m\times n}$, that is, $\Delta_3\left(H_{1\times  m\times n}\right)$, is a shellable simplicial complex. For this, our objective is to define an order on the facets of $\Delta_3\left(H_{1\times  m\times n}\right)$ and show that it is a shelling order. Before that, we need to specify the ordering on the vertices, and for this graph, we will follow the lexicographic ordering of the vertices of $H_{1\times  m\times n}$.

Before explicitly defining a shelling order on the facets of $\Delta_3\left(H_{1\times  m\times n}\right)$, let us find out the number of facets.

\subsection{Number of facets in $\Delta_3\left(H_{1\times  m\times n}\right)$}

Let $\eta_{m,n}$ be the number of facets of $\Delta_3\left( H_{1\times m \times n} \right)$. Precisely speaking, $\eta_{m,n}$ is the number of size-3 subsets of the set $[2m+2n+2mn]$ excluding those that induce $P_3$. Note that there are $\binom{2m+2n+2mn}{3}$ subsets of size 3 in $[2m+2n+2mn]$. We state the following lemma to find the number of induced $P_3$ in $H_{1\times  m\times n}$.

\begin{lemma}\label{lemma: number of induced P3 in hexagonal grid graph}
    For $m,n \geq 1$, let $\delta_{m,n}$ be the number of induced $P_3$ in $H_{1\times  m\times n}$. Then, $$\delta_{m,n} = 6mn+2m+2n-4.$$
\end{lemma}

\begin{proof}
    Observe that there are exactly $2mn - 2$ vertices of degree three (the vertices in the interior), each contributing to three induced $P_3$, and exactly $2m + 2n$ vertices of degree two (the vertices on the boundary), contributing to $2m + 2n + 2$ induced $P_3$'s. Thus, the total number of induced $P_3$ in $H_{1 \times m \times n}$ is $$ 3(2mn - 2) + (2m + 2n + 2) = 6mn + 2m + 2n - 4.$$
    
    This completes the proof of \Cref{lemma: number of induced P3 in hexagonal grid graph}.
\end{proof}

Hence, $$\eta_{m,n} = \binom{2m+2n+2mn}{3} - (6mn+2m+2n-4).$$ 

\subsection{Shelling order for the facets of $\Delta_3\left(H_{1\times  m\times n}\right)$}

%\subsection[Shelling order for the facets of $\Delta_3\left(H_{1\times  m\times n}\right)$]{Shelling order for the facets of $\texorpdfstring{\bm{\Delta_{3}\left(H_{1\times  m\times n}\right)}{\Delta_3\left(H_{1\times  m\times n}\right)}$}

We now give the idea for constructing an ordering to the facets of $\Delta_3\left(H_{1\times m \times n}\right)$, which we will prove is indeed a shelling ordering. First, arrange the facets of $\Delta_3 \left(H_{1\times m\times n}\right)$ in \textit{reverse lexicographic order}, then remove certain facets during the process and append them to the end of the ordering in the order in which they were removed. 

Let ``$\precc$" and ``$\succc$" denote the reverse and general lexicographic ordering, respectively.
Let $F_1', F_2',\dots,F_{\eta_{m,n}}'$ be the facets of $\Delta_3\left(H_{1\times m \times n}\right)$ where, \begin{align}\label{eqn: reverse lexico ordering}
    F_1' \precc  F_2' \precc \dots \precc F_{\eta_{m,n}}'.
\end{align}

For $m,n \geq 1$, we define $T_i$, for some $1 \leq i \leq \beta_{m,n}$, where $$\beta_{m,n} = \begin{cases}
    mn-2,& m\geq2\\
    n-1,& m = 1
\end{cases},$$ as follows.
\begin{enumerate}
    \item When $ m=1 $ and $n \geq 1$. For $1 \leq i \leq n-1$, we define $T_i$'s as $$ T_i = \lcu 2i + 2n,\ 2i+2n+2,\ 2i+2n+3\rcu^c.
    $$

    \item When $m \geq 2$ and $n \geq 1$.
    \begin{enumerate}
        \item For each $ k = 1, 2, \dots, n - 1 $, and for $(k-1)m + 1 \leq i \leq km$, define
        $$
        T_i = \lcu \begin{array}{c}
             i + m + n + mn + (k - 1),\ i + 2m + n + mn + k,\\
             i + 2m + n + mn + (k + 1)
        \end{array} \rcu^c.
        $$
    
        \item For $ (n-1)m + 1 \leq i \leq nm-2 $, define
        $$
        T_i = \left\{ i + m + 2n + mn,\; i + 2m + 2n + mn,\; i + 2m + 2n + mn + 1 \right\}^c.
        $$
    \end{enumerate}
\end{enumerate}

Let $t_i$ be the index of $T_i$ in \Cref{eqn: reverse lexico ordering}, i.e, $T_i = F_{t_i}'$. We remove the facets $T_1$, $T_2$, $\dots$, $T_{\beta_{m,n}}$ one by one from \Cref{eqn: reverse lexico ordering} and place them at the end in the order in which they were removed.

\begin{definition}\label{defn: new shelling order}
    We define the following ordering on the facets of $\Delta_3\left(H_{1\times  m\times n}\right)$: 
    \begin{align}\label{eqn: explicit shelling order}
        F_1,\ F_2, \dots,\ F_{\eta_{m,n}-\beta_{m,n}},\ T_1,\ T_2, \dots,\ T_{\beta_{m,n}}.
    \end{align}

    where the facets $F_j$ using \Cref{eqn: reverse lexico ordering} is given by:

    \begin{enumerate}
    \itemsep0em
        \item For $1 \leq j \leq t_1 - 1$, $$F_j = F_j'.$$
        \item For $2 \leq i \leq \beta_{m,n}-1$ and $t_{i-1}+1 \leq j \leq t_i - 1$, $$F_{j-i+1} = F_j'.$$
        \item For $t_{\beta_{m,n}}+1 \leq j \leq \eta_{m,n}$, $$F_{j-\beta_{m,n}} = F_j'.$$
    \end{enumerate}
\end{definition}

\begin{obs}\label{obs: Ti'c is neighborhood of certain vertices}
    Note that $T_i^c$ defined above are the neighborhoods of certain vertices, that is,
    \begin{align}\label{Eqn: nbhd description of Ti}
        T_i^c = \begin{cases}
            \mathcal{N} \lno i+m \rno, & 1 \leq i \leq m;\\
            \mathcal{N} \lno i+m+1 \rno, & m+1 \leq i \leq 2m;\\
            \mathcal{N} \lno i+m+2\rno, & 2m+1 \leq i \leq 3m;\\
            \vdots& \\
            \mathcal{N} \lno i+m+\lno k-1 \rno \rno, & \lno k-1\rno m + 1 \leq i \leq km;\\
            \vdots & \\
            \mathcal{N} \lno i+m+\lno n-2 \rno \rno, & \lno n-2 \rno m + 1 \leq i \leq \lno n-1 \rno m;\\
            \mathcal{N} \lno i+m+n \rno, & \lno n-1 \rno m + 1 \leq i \leq mn-2.\\
        \end{cases}
    \end{align}
\end{obs}

The central idea of this construction is to begin with the reverse lexicographic order on facets and then relocate a carefully chosen family of exceptional facets corresponding to neighborhoods of certain vertices. These exceptional facets obstruct shellability if left in place and must be appended at the end.

%For example, when $m=1$ and $n=5$, the $T_i$'s are\begin{align*}
%T_1 &= \{12, 14, 15\}, \\
%T_2 &= \{14, 16, 17\}, \\
%T_3 &= \{16, 18, 19\}, \\
%T_4 &= \{18, 20, 21\}.
%\end{align*}

%Similarly, when $m=4$ and $n=6$, the $T_i$'s are

%\begin{multicols}{2}
%\begin{itemize}
%\itemsep0em  
%%  \item $T_1 = \{20, 22, 23\}$
%  \item $T_2 = \{22, 24, 25\}$
%  \item $T_3 = \{24, 26, 27\}$
%  \item $T_4 = \{26, 28, 29\}$
%  \item $T_5 = \{28, 30, 31\}$
%  \columnbreak
%  \item $T_6 = \{30, 32, 33\}$
%  \item $T_7 = \{32, 34, 35\}$
% \item $T_8 = \{34, 36, 37\}$
%  \item $T_9 = \{36, 38, 39\}$
%  \item $T_{10} = \{38, 40, 41\}$
%\end{itemize}
%\end{multicols}

We call the vertex $c_i$ for which $T_i^c = \mathcal{N}\lno c_i\rno$ as the \textit{centre} of $T_i$, and observe that, $$ T_1 \precc  T_2 \precc \dots \precc T_{\beta_{m,n}}.$$ Also, we have some immediate observations from these explicit constructions of $T_i$.

\begin{obs}\label{obs: every element of T_i^c is in V2}
    $T_i^c \subset V_2$, for all $1 \leq i \leq \beta_{m,n}$.
\end{obs}

\begin{obs}\label{obs: Ti^c is an independent set}
    For all $1 \leq i \leq \beta_{m,n}$, $T_i^c$ is an independent set.
\end{obs}

\begin{obs}
    If $T_i^c = \lcu x_1,\ x_2,\ x_3 \rcu$, with $x_1 < x_2 < x_3$, then $x_3 = x_2 + 1$, that is, $x_2$ and $x_3$ are consecutive integers. Consequently, $$T_i = \lcu x_1,\ x_2,\ x_2+1 \rcu.$$
\end{obs}

Before proving our main result, we prove the following lemma, which will play a significant role in proving the shellability of $\Delta_3 \left(H_{1\times m \times n}\right)$.

\begin{lemma}\label{lemma: Ti fails in reverse lex ordering}
    Let $T_1,\ T_2,\dots,\ T_{\beta_{m,n}}$ denote the facets that were removed from the reverse lexicographic ordering of the facets of $\Delta_3(H_{1 \times m \times n})$, as defined previously. If a facet $T_i$ were kept at its original position, say $j_i$ (that is, $F_{j_i} = T_i$), in the reverse lexicographic order, then the shellability condition would fail for each such $T_i$; that is, there exists an index $i' < j_i$ such that there is no index $r < j_i$ satisfying
    $$F_r \cap F_{j_i} = F_{j_i} \setminus \{ \lambda \},$$ where $\lambda \in F_{j_i} \setminus F_{i'}$.
\end{lemma}

\begin{proof}
    Fix a facet $T_i$. By construction, $T_i$ is the open neighborhood of a vertex $x_i$, as defined in \Cref{Eqn: nbhd description of Ti}. For $T_i$ to satisfy the shellability condition \hyperref[shellability condition]{($*$)}, there must exist, for every $i' < j_i$, a facet $F_r$ with $r < j_i$ satisfying
    $$F_r \cap F_{j_i} = F_{j_i} \setminus \{ \lambda \},$$ for some $\lambda \in F_{j_i} \setminus F_{i'}$. We claim this is impossible because an index $i' < j_i$ exists for which no such $F_r$ exists. Define:
    $$F_{i'} := \left( \{ x_i,\ 2m+2n+2mn-1,\ 2m+2n+2mn \} \right)^c.$$
    
    Observe that
    $$T_i \setminus F_{i'} =
    \begin{cases}
    F_{i'}^c, & \text{for } 1 \leq i < \beta_{m,n}; \\
    \{ x_{\beta_{m,n}},\ 2m+2n+2mn \}, & \text{for } i = \beta_{m,n}.
    \end{cases}$$

    To construct $F_r$, we must remove a vertex from $T_i^c$ and add a vertex from $T_i \setminus F_{i'}$. However, for any $\lambda \in T_i \setminus F_{i'}$, the following occurs:
    \begin{enumerate}[label=\roman*.]
        \item If $\lambda = x_i$, then $F_r^c$ induces a $P_3$, contradicting the definition of a facet in $\Delta_3 \lno H_{1\times m \times n} \rno$.
        \item If $\lambda \neq x_i$, then $T_i^c \succc F_r^c$, or equivalently, $T_i \precc F_r$, which contradicts $r < j_i$ under reverse lexicographic ordering.
    \end{enumerate}

    In both cases, no such $F_r$ satisfying the shellability condition exists, and thus the condition fails for $T_i$ at position $j_i$. This completes the proof of \Cref{lemma: Ti fails in reverse lex ordering}.
\end{proof}

\begin{remark}
    There are other possible choices for the facet $F_{i'}$ which also fail to work for each $T_i$. In each such case, the corresponding $F_r$ constructed from that choice leads to the same issues described in the proof above, and the argument for the failure remains the same.
\end{remark}

\subsection[Shellability of $\Delta_3\lno H_{1\times m \times n}\rno$]{Shellability of $\texorpdfstring{\bm{\Delta_{3}\lno H_{1\times m \times n}\rno}}{\Delta_{3}\lno H_{1\times m \times n}\rno}$}

We now state the main result of this section.

\begin{theorem}\label{thm: shellability of hexagonal grid graphs}
    For $m, n\geq 1$, the simplicial complex $\Delta_3\left(H_{1\times m \times n}\right)$ is shellable.
\end{theorem}

\begin{proof}
    We claim that the ordering given in \Cref{defn: new shelling order} is the required shelling order for the facets of $\Delta_3\left(H_{1\times m \times n}\right)$. That is, we aim to show the following:
    \begin{enumerate}[label = \Alph*.]
        \item\label{part A} $F_1,\ F_2,\dots,\ F_{\eta_{m,n} - \beta_{m,n}}$ are in shelling order.
        \item\label{part B} If the $T_i$'s are placed among the $F_i$'s according to the reverse lexicographic ordering (that is, \Cref{eqn: reverse lexico ordering}), the resulting order is not a shelling order.
        \item\label{part C} For all $1 \leq i \leq \beta_{m,n}$, $T_i$ satisfies the shellability condition with every $F_j$, where $1\leq j \leq \eta_{m,n} - \beta_{m,n}$ and with every $T_{i'}$, where $i' < i$.
        \item\label{part D} $T_1,\ T_2, \dots,\ T_{\beta_{m,n}}$ are in shelling order.
    \end{enumerate}

    \hyperref[part B]{B} follows from \Cref{lemma: Ti fails in reverse lex ordering}. In its proof, the failure of $T_i$ to satisfy the shellability condition when kept in reverse lexicographic ordering arose because the required $F_r$ (with $r<j_i$) appeared after $T_i$. Placing all $T_i$ at the end resolves this issue. Thus, $T_i$ satisfies the shellability condition with $F_j$. Consequently, \hyperref[part D]{D} is proved. It remains only to show in \hyperref[part D]{D} that $T_i$ satisfies the shellability condition with $T_{i'}$ with $i' < i$. The argument for this is analogous to that used in proving \hyperref[part A]{A}. We therefore proceed to prove \hyperref[part A]{A}.
    
    To prove \hyperref[part A]{A}, we need to prove that the ordering $F_1,\ F_2,\dots,\ F_{\eta_{m,n} - \beta{m,n}}$ satisfies the condition \Cref{shellability condition}. In other words, we must construct a suitable $F_r$ for any given choice of $F_i$ and $F_j$. Note here that, $i<j$ if and only if $F_i \precc F_j$ (or equivalently, $F_i^c \succc F_j^c$). We will deal with the complements of the facets to prove our result. Also, observe that for any $1 \leq i < j \leq \lno \eta_{m,n} - \beta_{m,n} \rno$, we have the following three scenarios:
    \begin{enumerate}[label = \arabic*.]
        \item $\left| F_i^c \cap F_j^c \right| = 0$;
        \item $\left| F_i^c \cap F_j^c \right| = 1$;
        \item $\left| F_i^c \cap F_j^c \right| = 2$.
    \end{enumerate}
    
    We deal with all of these cases one by one.
    
    %% Case 1: | Fi^c ∩ Fj^c | = 0 %%
    \caseheading{Case 1:}\label{Case 1 - main theorem} $\bm{\left| F_{i}^c \cap F_{j}^c \right| = 0}$.
        
    Let $F_i^c = \lcu i_1,\ i_2,\ i_3 \rcu$ and $F_j^c = \lcu j_1,\ j_2,\ j_3 \rcu$. Without loss of generality (WLOG), let $i_1 < i_2 < i_3$ and $j_1 < j_2 < j_3$.
        
    Since $F_i^c \succc F_j^c$, we must have $i_1 < j_1$. Note that $F_i^c \subset F_j$ (since $F_i^c \cap F_j^c = \emptyset$), which implies, $$F_j\setminus F_i =  F_i^c.$$
        
    We examine all possible configurations of $F_i^c$ and $F_j^c$ under the given condition, and construct $F_r^c$ accordingly. In this case, $F_r$ is explicitly defined as, $$F_r = \left(F_j \cup \lcu \alpha \rcu\right) \setminus\lcu \lambda \rcu,$$ where, $\lambda \in F_i^c$ and $\alpha \in F_j^c$. Equivalently, $$F_r^c = \left(F_j^c \cup \lcu \lambda \rcu \right) \setminus \lcu \alpha \rcu.$$
        
    We call the pair $\lno \lambda, \alpha \rno$ an \textit{optimal pair} if it yields such an $F_r^c$ that satisfies \hyperref[shellability condition]{$(*)$}. Based on the relative ordering of $i_2$, $i_3$, $j_1$, $j_2$ and $j_3$, we have the several sub-cases. The objective here is to look at the worst-case scenarios for each sub-case, that is, if $F_r^c$ is some $T_i^c$ or induces $P_3$.
    
    \begin{enumerate}[label = \bfseries1.\arabic*.]
    
        %% SUBCASE 1.1 %%
        \item\label{subcase 1.1} $\bm{i_{1} < i_{2} < i_{3} < j_{1} < j_{2} < j_{3}.}$

        To find an optimal pair, we start by choosing $\lno \lambda, \alpha \rno = \lno i_1, j_2 \rno$. This gives $$F_r^c = \lcu i_1 < j_1 < j_3 \rcu,$$ such that $F_r^c \succc F_j^c$.

        At first, suppose that $F_r^c = T_k^c$, for some $1 \leq k \leq \beta_{m,n}$. This will imply that $j_3 = j_1 + 1$, that is, $j_1$ and $j_3$ are consecutive integers. However, this is not possible since $j_1 < j_2 < j_3$. Therefore, $F_r^c \neq T_k^c$, for all $1 \leq k \leq \beta_{m,n}$.

        Now, if $F_r^c$ induces $P_3$, then by \Cref{Observation 1}, one of the following two configurations must occur:
        
        \begin{enumerate}[label = (\alph*)]
            \item If $i_1, j_1 \in V_1$ and $j_3 \in V_2$, then $i_1$ and $j_1$ are the endpoints, each adjacent to $j_3$, the midpoint.
            \item If $i_1 \in V_1$ and $j_1, j_3 \in V_2$, then $j_1$ and $j_3$ are the endpoints, each adjacent to $i_1$, the midpoint.
        \end{enumerate}

        It turns out that both of these cases are possible. For instance, in the $\Delta_3 \lno H_{1\times 4 \times 6} \rno$ (see \Cref{fig: graph of H(1_4_6)}), if we take $F_i^c = \lcu 11, 12, 16 \rcu$ and $F_j^c = \lcu 17, 46, 47\rcu$, then for the given choice of $\lambda$ and $\alpha$ we obtain $F_r^c = \lcu 11, 17, 47 \rcu$, which indeed forms an induced $P_3$ in accordance with condition in (a). Similarly, choosing $F_i^c = \lcu 17, 18, 46 \rcu$ and $F_j^c = \lcu 47, 48, 52\rcu$, yields $F_r^c = \lcu 17, 47, 52 \rcu$, which again gives an induced $P_3$ as required by condition in (b). Consequently, the choice $\lno \lambda, \alpha \rno = \lno i_1, j_2 \rno$ does not always yield an optimal pair, as demonstrated above.

        Since the previous choice of $\lno \lambda, \alpha \rno$ is not optimal, we retain the same positioning of $i_1$, $i_2$, $i_3$, $j_1$, $j_2$, and $j_3$ but replace $\alpha = j_2$ by $\alpha = j_3$, while keeping $\lambda = i_1$. Thus, we have $$F_r^c = \lcu i_1 < j_1 < j_2 \rcu,$$ such that $F_r^c \succc F_j^c$. We now establish that this selection is indeed optimal for both conditions given in (a) and (b). In particular, for each part, we will verify both possible worst-case scenarios, namely when $F_r^c$ coincides with some $T_k^c$ and when it induces $P_3$.

        For part (a), $F_r^c$ cannot coincide with any $T_k^c$. Indeed, if this were the case, then $F_r^c$ would have $j_3$ as its center, which is impossible since $j_3 \in V_2$. By \Cref{obs: V1 nbhd c} and \Cref{obs: every element of T_i^c is in V2}, the center of any $T_k^c$ must lie in $V_1$. On the other hand, if $F_r^c$ induces $P_3$, then this can only occur if $j_2$ is in the position of $j_3$, that is, $j_2 = j_3$, which is not the case here. Hence, the pair $\lno \lambda, \alpha \rno = \lno i_1, j_3 \rno$ is optimal.
        
        For part (b), $F_r^c$ cannot be a $T_k^c$ because $i_1$ and $j_1$ are adjacent, whereas by \Cref{obs: Ti^c is an independent set}, all $T_k^c$ are independent. If instead $F_r^c$ induces a $P_3$, then $j_2$ must be adjacent to either $i_1$ or $j_1$. It cannot be adjacent to $i_1$, for in that case $F_j^c = \mathcal{N}\lno i_1 \rno$, which is precisely of the form of some $T_k^c$. However, this is a contradiction to $F_j^c$ not being any $T_i^c$. 
        
        At this point, we explicitly mention a possible exception to this argument. Let $z = 2m+2n+2mn$ and take $F_j^c = \lcu z - m - 1, z-1, z \rcu$. Clearly, $F_j^c = \mathcal{N}\lno m+n+mn-1 \rno$ and it does not coincide with any $T_k^c$, as is clear from the construction of $T_k^c$. In this particular case, both previous choices of $\lno \lambda, \alpha \rno$ will satisfy the condition of all possible worst-case scenarios, including part (b). However, this can be easily resolved by replacing the choice of $\lambda = i_1$ with $\lambda = i_2$ (or even $i_3$), while keeping $\alpha = j_2$. Thus, we will also obtain an optimal pair for this exception.
        
        Now, if $j_2$ is adjacent to $j_1$, then $j_2\in V_1$, since  $j_1 \in V_2$ (by \Cref{obs: V1 nbhd c}, which forces $j_2 < j_1$, which is not possible here. Therefore, the pair $\lno \lambda, \alpha \rno = \lno i_1, j_3 \rno$ (or $\lno i_2, j_3 \rno$ for the exceptional case) forms an optimal pair.

        %% SUBCASE 1.2 %%
        \item\label{subcase 1.2} $\bm{i_{1} < i_{2} < j_{1} < i_{3} < j_{2} < j_{3}.}$

        Observe that $\lno \lambda, \alpha \rno = \lno i_3, j_1 \rno$ cannot be an optimal pair, since it will imply $F_j^c \succc F_r^c$, or equivalently $F_r \precc F_j$, contrary to our requirement. This subcase then follows using a similar argument given in subcase \hyperref[subcase 1.1]{1.1}.

        %% SUBCASE 1.3 %%
        \item\label{subcase 1.3} $\bm{i_{1} < i_{2} < j_{1} < j_{2} < i_{3} < j_{3}.}$

        Observe that $\lno \lambda, \alpha \rno$ cannot be $\lno i_3, j_1 \rno$ or $\lno i_3, j_2 \rno$ as a choice for the optimal pair, since this forces $F_j^c \succc F_r^c$. The same optimal pairs that are defined in subcase \hyperref[subcase 1.1]{1.1}, and the arguments thereafter, follow here too, with only a modification concerning the exceptional configuration. Earlier, the exception arises when $F_j^c = \lcu z-m-1,\, z-1,\, z \rcu$, with $j_2 = z-1$ and $j_3 = z$, that is, there is no integer between $j_2$ and $j_3$. However, in this subcase, we have $j_2 < i_3 < j_3$, so this situation cannot arise.
        
        %% SUBCASE 1.4 %%
        \item\label{subcase 1.4} $\bm{i_{1} < i_{2} < j_{1} < j_{2} < j_{3} < i_{3.}}$

        Observe that $\lno \lambda, \alpha \rno$ cannot be any of $\lno i_3, j_1 \rno$, $\lno i_3, j_2 \rno$, or $\lno i_3, j_3 \rno$ as a choice for the optimal pair, since this forces $F_j^c \succc F_r^c$. The same optimal pairs that are defined in subcase \hyperref[subcase 1.1]{1.1}, and the arguments thereafter, follow here too, with some changes in the exceptional configuration. In subcase \hyperref[subcase 1.1]{1.1}, the exception arises when $F_j^c = \lcu z-m-1,, z-1, z \rcu$ with $j_3 = z = 2m+2n+2mn$. In the present subcase, however, this situation does not arise because we require $j_3 < i_3$, which leaves no possible choice for $i_3$ in $[2m+2n+2mn]$.
        
        %% SUBCASE 1.5 %%
        \item\label{subcase 1.5} $\bm{i_{1} < j_{1} < i_{2} < i_{3} < j_{2} < j_{3.}}$

        Observe that $\lno \lambda, \alpha \rno$ cannot be $\lno i_3, j_1 \rno$ or $\lno i_3, j_2 \rno$ as a choice for the optimal pair, since this forces $F_j^c \succc F_r^c$. The same optimal pairs that are defined in subcase \hyperref[subcase 1.1]{1.1}, and the arguments thereafter, follow here too, with some changes in the exceptional configuration. If $m \geq 3$, then we can proceed with the argument given for the exception in subcase \hyperref[subcase 1.1]{1.1}. If $m=2$, then $F_j^c = \lcu z-3, z-1, z \rcu$ implying there is only vertex possible in between $j_1 = z-3$ and $j_2 = z-1$ which is $z-2$. However, in this subcase, we have $j_1 < i_2 < i_3 < j_2$, implying this situation cannot arise here.
        
        %% SUBCASE 1.6 %%
        \item\label{subcase 1.6} $\bm{i_{1} < j_{1} < i_{2} < j_{2} < i_{3} < j_{3.}}$

        Observe that $\lno \lambda, \alpha \rno$ cannot be any of $\lno i_2, j_1 \rno$, $\lno i_3, j_1 \rno$, or $\lno i_3, j_2 \rno$ as a choice for the optimal pair, since this forces $F_j^c \succc F_r^c$. This subcase then follows using a similar argument given in subcase \hyperref[subcase 1.1]{1.1} and \hyperref[subcase 1.3]{1.3}.
        
        %% SUBCASE 1.7 %%
        \item\label{subcase 1.7} $\bm{i_{1} < j_{1} < i_{2} < j_{2} < j_{3} < i_{3.}}$

        Observe that $\lno \lambda, \alpha \rno$ cannot be any of $\lno i_2, j_1 \rno$, $\lno i_3, j_1 \rno$, $\lno i_3, j_2 \rno$, or $\lno i_3, j_3 \rno$ as a choice for the optimal pair, since this forces $F_j^c \succc F_r^c$. This subcase then follows using a similar argument given in subcase \hyperref[subcase 1.1]{1.1} and \hyperref[subcase 1.4]{1.4}.

        %% SUBCASE 1.8 %%
        \item\label{subcase 1.8} $\bm{i_{1} < j_{1} < j_{2} < i_{2} < i_{3} < j_{3.}}$

        Observe that $\lno \lambda, \alpha \rno$ cannot be any of $\lno i_2, j_1 \rno$, $\lno i_2, j_2 \rno$, $\lno i_3, j_1 \rno$, or $\lno i_3, j_2 \rno$ as a choice for the optimal pair, since this forces $F_j^c \succc F_r^c$. This subcase then follows using a similar argument given in subcase \hyperref[subcase 1.1]{1.1} and \hyperref[subcase 1.3]{1.3}.
        
        %% SUBCASE 1.9 %%
        \item\label{subcase 1.9} $\bm{i_{1} < j_{1} < j_{2} < i_{2} < j_{3} < i_{3.}}$

        Observe that $\lno \lambda, \alpha \rno$ cannot be any of $\lno i_2, j_1 \rno$, $\lno i_2, j_2 \rno$, $\lno i_3, j_1 \rno$, $\lno i_3, j_2 \rno$, or $\lno i_3, j_3 \rno$ as a choice for the optimal pair, since this forces $F_j^c \succc F_r^c$. This subcase then follows using a similar argument given in subcase \hyperref[subcase 1.1]{1.1} and \hyperref[subcase 1.4]{1.4}.

        %% SUBCASE 1.10 %%
        \item\label{subcase 1.10} $\bm{i_{1} < j_{1} < j_{2} < j_{3} < i_{2} < i_{3.}}$

        Observe that $\lno \lambda, \alpha \rno$ is an optimal pair only when $\lambda = i_1$ for some $\alpha \in F_j^c$; otherwise, we have $F_j^c \succc F_r^c$. This subcase then follows using a similar argument given in subcase \hyperref[subcase 1.1]{1.1} and \hyperref[subcase 1.4]{1.4}.
    \end{enumerate}

    %% Case 2: | Fi^c ∩ Fj^c | = 1 %%
    \caseheading{Case 2:}\label{Case 2 - main theorem} $\bm{\left| F_{i}^{c} \cap F_{j}^{c} \right| = 1}$.
        
    Let $F_i^c = \lcu i_1,\ i_2,\ x \rcu$ and $F_j^c = \lcu j_1,\ j_2,\ x \rcu$. WLOG, let $i_1 < i_2$ and $j_1 < j_2$.
        
    For the ordering $F_i^c \succc F_j^c$ to hold, it is necessary that $i_1 < j_1$. This will be easily verified by a case-by-case analysis of all possible orderings of $x,\ i_1,\ i_2,\ j_1,\ j_2$, which we will present in the subsequent discussion. Moreover, observe here that $\lcu i_1, i_2 \rcu \subset F_j$, implying, $$F_j \setminus F_i = \lcu i_1, i_2 \rcu.$$ Thus, for this case, we define $F_r$ explicitly as, $$F_r = \lno F_j \cup \lcu \alpha \rcu \rno \setminus \lcu \lambda \rcu,$$ or equivalently, $$F_r^c = \lno F_j^c \cup \lcu \lambda \rcu \rno \setminus \lcu \alpha \rcu,$$ where, $\lambda \in \lcu i_1, i_2 \rcu$ and $\alpha \in F_j^c$. 
        
    If we fix the position of $x$ in $F_i^c$ and accordingly look at the possibilities of $F_j^c$, then we will have the following subcases:
    
    \begin{enumerate}[label = \bfseries2.\arabic*.]
        %% SUB-CASE 2.1 - Main Theorem %%
        \item\label{subcase 2.1} $\bm{F_{i}^{c} = \left\{ x < i_{1} < i_{2} \right\}}$ \textbf{and} $\bm{F_{j}^{c} = \lcu x < j_{1} < j_{2} \rcu}$.
        
        It is clear from the above configuration of $F_i^c$ and $F_j^c$ that in order to have $F_i^c \succc F_j^c$, we must have $i_1 < j_1$. Observe that $\lno \lambda, \alpha \rno$ cannot be $\lno i_1, x \rno$ or $\lno i_2, x \rno$ as a choice for the optimal pair, since it forces $F_j^c \succc F_r^c$.
        
        Moreover, it follows that $F_j^c$ can only assume the form $\lcu x < j_1 < j_2 \rcu$ for the given $F_i^c$. Indeed, if $F_j^c$ is of the form $\lcu j_1 < x < j_2 \rcu$ or $\lcu j_1 < j_2 < x \rcu$, then we will always have $F_j^c \succc F_i^c$, contrary to our requirement. So, it remains to analyze the possible orderings of $i_2$, $j_1$, and $j_2$, leading to the following scenarios.

        \begin{enumerate}[label = \bfseries2.1.\arabic*.]
            %% 2.1.1 %%
            \item\label{subcase 2.1.1} \textbf{$\bm{i_{2} < j_{1} < j_{2}}$ (that is, $\bm{ x < i_{1} < i_{2} < j_{1} < j_{2}}$).}

            We proceed by choosing $(\lambda, \alpha) = (i_1, j_1)$ as our initial choice for the optimal pair. The $F_r^c$ thus obtained will not assume the form of $T_k^c$ for any $1 \leq k \leq \beta_{m,n}$, and the argument follows similar steps to that given in subcase \hyperref[subcase 1.1]{1.1}.

            If $F_r^c$ induces $P_3$, then, as in subcase \hyperref[subcase 1.1]{1.1}, we obtain two possible scenarios analogous to parts (a) and (b). Both of these can occur in the present subcase. We will not provide a concrete example, since it can be easily constructed by analyzing $F_i^c$ and $F_j^c$. For completeness, we will explicitly state parts (a) and (b) in this case as well, since doing so will make the later arguments easier to understand.

            \begin{enumerate}[label = (\alph*)]
                \item If $x, i_1 \in V_1$ and $j_2 \in V_2$, then $x$ and $i_1$ are the endpoints, each adjacent to $j_2$, the midpoint.
                \item If $x \in V_1$ and $i_1, j_2 \in V_2$, then $i_1$ and $j_2$ are the endpoints, each adjacent to $x$, the midpoint.
            \end{enumerate}

            Since the initial choice $(\lambda, \alpha) = (i_1, j_1)$ was not optimal, we now replace $\alpha = j_1$ with $\alpha = j_2$, keeping $\lambda = i_1$ and the positioning of $x,\ i_1,\ i_2,\ j_1,$ and $j_2$ unchanged in both part (a) and (b). This will give $$F_r^c = \lcu x < i_1 < j_1 \rcu.$$
            
            For parts (a) and (b), we verify that the worst-case scenarios, namely, when $F_r^c$ coincides with some $T_k^c$ and when it induces a $P_3$, cannot occur. The argument for part (a) here is analogous to that in subcase \hyperref[subcase 1.1]{1.1}. We modify the argument for part (b) here.
            
            For part (b), $F_r^c$ cannot be a $T_k^c$ because $x$ and $i_1$ are adjacent, whereas by \Cref{obs: Ti^c is an independent set}, all $T_k^c$ are independent. If instead $F_r^c$ induces a $P_3$, then $j_1$ must be adjacent to either $x$ or $i_1$. It cannot be adjacent to $x$, since that would make $F_j^c$ induce a $P_3$, contradicting the disconnectedness of $F_j^c$. If $j_1$ is adjacent to $i_1$, then $j_1 \in V_1$, since $i_1\in V_2$ (by \Cref{obs: V1 nbhd c}), which forces $j_1 < i_1$, which is not possible here. Also, we do not get any exceptions here as we do in subcase \hyperref[subcase 1.1]{1.1}. Therefore, the pair $\lno \lambda, \alpha \rno$ is optimal.
            
            %% 2.1.2 %%
            \item\label{subcase 2.1.2} \textbf{$\bm{j_{1} < i_{2} < j_{2}}$ (that is, $\bm{x < i_{1} < j_{1} < i_{2} < j_{2}}$).}

            Observe that the choice $\lno \lambda, \alpha \rno  = \lno i_2, j_1 \rno$ will never be an optimal pair, since it forces $F_j^c \succc F_r^c$. This subcase follows from the previous subcase.
            
            %, that is, begin with $\lambda = i_1$ and $\alpha = j_1$, and whenever that assignment fails, replace $\alpha = j_1$ by $\alpha = j_2$ while keeping $\lambda = i_1$. The remainder of the argument then follows exactly as in the first part, with these substitutions. In the previous case, we provided concrete counterexamples; with the substitutions indicated here, one can analogously construct similar counterexamples. We omit them, as they follow directly from the earlier ones.

            %% 2.1.3 %%
            \item\label{subcase 2.1.3} \textbf{$\bm{j_{1} < j_{2} < i_{2}}$ (that is, $\bm{x < i_{1} < j_{1} < j_{2} < i_{2}}$).}

            Observe that $\lno \lambda, \alpha \rno$ cannot be $\lno i_2, j_1 \rno$ or $\lno i_2, j_2 \rno$ as a choice for the optimal pair, since it forces $F_j^c \succc F_r^c$. The argument then proceeds exactly as in the subcase \hyperref[subcase 2.1.1]{2.1.1}.
        \end{enumerate}

        %% SUBCASE 2.2 - Main Theorem %%
        \item\label{subcase 2.2} \textbf{$\bm{F_{i}^{c} = \lcu  i_{1} < x < i_{2} \rcu}$ and $\bm{F_{j}^{c} = \lcu x < j_{1} < j_{2} \rcu}$.}

        Clearly, $i_1 < x < j_1$ will always hold and thus $F_i^c \succc F_j^c$. Also, observe that $\lno \lambda, \alpha \rno$ cannot be $\lno i_2, x \rno$ as a choice for the optimal pair, since it forces $F_j^c \succc F_r^c$. We only have to look at the different ordering of $i_2$, $j_1$, and $j_2$. This can be done in the following ways:

        \begin{enumerate}[label = \bfseries2.2.\arabic*.]
            %% 2.2.1 %%
            \item\label{subcase 2.2.1} \textbf{$\bm{i_{2} < j_{1} < j_{2}}$ (that is, $\bm{i_{1} < x < i_{2} < j_{1} < j_{2}}$).}

            This case is analogous to subcase \hyperref[subcase 1.1]{1.1}, including the exception we get there, and unlike subcase \hyperref[subcase 2.1.1]{2.1.1}, where we do not.

            %% 2.2.2 %%
            \item\label{subcase 2.2.2} \textbf{$\bm{j_{1} < i_{2} < j_{2}}$ (that is, $\bm{i_{1} < x < j_{1} < i_{2} < j_{2}}$).}

            Observe that $\lno \lambda, \alpha \rno$ cannot be $\lno i_2, j_1 \rno$ a choice for the optimal pair, since it forces $F_j^c \succc F_r^c$. This subcase proceeds analogously to subcase \hyperref[subcase 1.1]{1.1}, but without any exception, and the reason for this follows from subcase \hyperref[subcase 1.3]{1.3}.
            
           % Clearly, the choice $\lambda = i_2$ and $\alpha = j_1$ is not possible, as it will force $F_j^c \succc F_r^c$. Instead, we adopt the choice of $\lambda$ and $\alpha$ from subcase \hyperref[sub-case 2.2.2]{A(2)}, and, following the argument in \hyperref[sub-case (2A)-1]{A(1)}, we obtain the desired optimal choice. However, some minor adjustments are needed in part (b) of the argument given in \hyperref[sub-case (2A)-1]{A(1)}, which we now discuss explicitly.  

        %First, we replace the initial choice $\alpha = j_1$ by $\alpha = j_2$, while keeping $\lambda = i_1$, to obtain $$F_r^c = \{\, i_1 < x < j_1 \,\}.$$
        %For part (b), $F_r^c$ cannot be a $T_k^c$ because $i_1$ and $x$ are adjacent. On the other hand, if $F_r^c$ induces a $P_3$, then $j_1$ must be adjacent to either $i_1$ or $x$. If $j_1$ is adjacent to $i_1$, then $j_2 = j_1+1$, which contradicts $j_1 < i_2 < j_2$. If $j_1$ is adjacent to $x$, then $j_1 \in V_1$, and by \Cref{Observation 1} (together with $i_1 \in V_2$) this would force $j_1 < i_1$, which is not possible.

            %% 2.2.3 %%
            \item\label{subcase 2.2.3} \textbf{$\bm{j_{1} < j_{2} < i_{2}}$ (that is, $\bm{i_{1} < x < j_{1} < j_{2} < i_{2}}$).}

            Observe that $\lno \lambda, \alpha \rno$ cannot be $\lno i_2, j_1 \rno$ or $\lno i_2, j_2 \rno$ as a choice for the optimal pair, since it forces $F_j^c \succc F_r^c$. This subcase proceeds analogously to subcase \hyperref[subcase 1.1]{1.1}, but without any exception, and the reason for this follows from subcase \hyperref[subcase 1.4]{1.4}.
        \end{enumerate}

        %% SUBCASE 2.3 - Main Theorem%%
        \item \textbf{$\bm{F_{i}^{c} = \lcu  i_{1} < x < i_{2} \rcu}$ and $\bm{F_{j}^{c} = \lcu j_{1} < x < j_{2} \rcu}$.}

        In order to have $F_i^c \succc F_j^c$, we must $i_1 < j_1 < x$ here. Similar to the subcase \hyperref[subcase 2.1]{2.1}, $\lno \lambda, \alpha \rno$ cannot be $\lno i_2, x \rno$ as a choice for the optimal pair, since it forces $F_j^c \succc F_r^c$. Therefore, we need to examine the orderings of $i_2$ and $j_2$. This can be done in the following ways,

        \begin{enumerate}[label = \bfseries2.3.\arabic*.]
            %% 2.3.1 %%
            \item\label{subcase 2.3.1} \textbf{$\bm{i_{2} < j_{2}}$ (that is, $\bm{i_{1} < j_{1} < x < i_{2} < j_{2}}$).}

            Observe that $\lno \lambda, \alpha \rno$ cannot be $\lno i_2, j_1 \rno$ as a choice for the optimal pair, since it forces $F_j^c \succc F_r^c$. This subcase proceeds analogously to subcase \hyperref[subcase 1.1]{1.1}, but without any exception, and the reason for this follows from subcase \hyperref[subcase 1.3]{1.3}.
            
            %% 2.3.2 %%
            \item\label{subcase 2.3.2} \textbf{$\bm{j_{2} < i_{2}}$ (that is, $\bm{i_{1} < j_{1} < x < j_{2} < i_{2}}$).}
            
            Observe that $\lno \lambda, \alpha \rno$ cannot be $\lno i_1, x \rno$ or $\lno i_2, x \rno$ as a choice for the optimal pair, since it forces $F_j^c \succc F_r^c$. Again, this subcase proceeds analogously to subcase \hyperref[subcase 1.1]{1.1}, but without any exception, and the reason for this follows from subcase \hyperref[subcase 1.4]{1.4}. The only modification is in the choice of the pair $\lno \lambda, \alpha \rno$, where the initial choice is $\lno i_1, x \rno$, and if this fails, we then choose $\lno i_1, j_2 \rno$.
        \end{enumerate}

        %% SUBCASE 2.4 %%c
        \item\label{subcase 2.4} \textbf{$\bm{F_{i}^{c} = \lcu  i_{1} < x < i_{2} \rcu}$ and $\bm{F_{j}^{c} = \lcu j_{1} < j_{2} < x\rcu}$.}

        Clearly, $i_1 < j_1 < j_2 < x < i_2$ is the only possibility here. Also, $\lno \lambda, \alpha \rno$ cannot be any of $\lno i_2, j_1 \rno$, $\lno i_2, j_2 \rno$, or $\lno i_1, x \rno$, as a choice for the optimal pair, since this forces $F_j^c \succc F_r^c$. This subcase proceeds analogously to subcase \hyperref[subcase 1.1]{1.1}, but without any exception, and the reason for this follows from subcase \hyperref[subcase 1.4]{1.4}. The only modification is in the choice of the pair $\lno \lambda, \alpha \rno$, where the initial choice is $\lno i_1, j_2 \rno$, and if this fails, we then choose $\lno i_1, x \rno$.

        %% SUBCASE 2.5 %%
        \item\label{subcase 2.5}\textbf{$\bm{F_{i}^{c} = \lcu  i_{1} < i_{2} < x \rcu}$ and $\bm{F_{j}^{c} = \lcu x < j_{1} < j_{2} \rcu}$}.

        Clearly, $i_1 < i_2 < x < j_1 < j_2$ is the only possibility here. This subcase follows analogously from the subcase \hyperref[subcase 1.1]{1.1}, including the exception we get there.
        
        %% SUBCASE 2.6 %%
        \item\label{subcase 2.6} \textbf{$\bm{F_{i}^{c} = \lcu  i_{1} < i_{2} < x \rcu}$ and $\bm{F_{j}^{c} = \lcu j_{1} < x < j_{2} \rcu}$.}

        In order to have $F_i^c \succc F_j^c$, we must have $i_1 < j_1$, and $i_2 < x < j_2$ will always hold. So, we only need to see how $i_2$ and $j_1$ are placed. This can be done in the following ways:

        \begin{enumerate}[label = \bfseries2.6.\arabic*.]
            %% 2.6.1 %%
            \item\label{subcase 2.6.1} \textbf{$\bm{i_{2} < j_{1}}$ (that is, $\bm{i_{1} < i_{2} < j_{1} < x < j_{2}}$).}

            This subcase follows analogously to the subcase \hyperref[subcase 1.1]{1.1}, including the exception we get there. The only modification is in the choice of the pair $\lno \lambda, \alpha\rno$, where the initial choice is $\lno i_1, x \rno$, and if it fails, we then choose $\lno i_1, j_2 \rno$

            %% 2.6.2 %%                
            \item\label{subcase 2.6.2} \textbf{$\bm{j_{1} < i_{2}}$ (that is, $\bm{i_{1} < j_{1} < i_{2} < x < j_{2}}$).}

            This case follows the exact same argument given for the previous subcase \hyperref[subcase 2.6.1]{2.6.1}.                
        \end{enumerate}

        %% SUBCASE 7 %%
        \item\label{subcase 2.7} \textbf{$\bm{F_{i}^{c} = \lcu  i_{1} < i_{2} < x \rcu}$ and $\bm{F_{j}^{c} = \lcu j_{1} < j_{2} < x \rcu}$.}

        In order to have $F_i^c \succc F_j^c$, we must have $i_1 < j_1$ here. Also, the position of $x$ is fixed, so we only need to see how $i_2$, $j_1$, and $j_3$ are arranged. We have the following possible scenarios, 

        \begin{enumerate}[label = \bfseries2.7.\arabic*.]
            %% 2.7.1 %%
            \item\label{subcase 2.7.1} \textbf{$\bm{i_{2} < j_{1} < j_{2}}$ (that is, $\bm{i_{1} < i_{2} < j_{1} < j_{2} < x}$).}
            
            This subcase follows analogously to the subcase \hyperref[subcase 1.1]{1.1}, including the exception we get there. The only modification is in the choice of the pair $\lno \lambda, \alpha\rno$, where the initial choice is $\lno i_1, j_2 \rno$, and if it fails, we then choose $\lno i_1, x \rno$
            
            %% 2.7.2 %%
            \item\label{subcase 2.7.2} \textbf{$\bm{j_{1} < i_{2} < j_{2}}$ (that is, $\bm{i_{1} < j_{1} < i_{2} < j_{2} < x}$).}

            This case follows the exact same argument given for the previous subcase \hyperref[subcase 2.7.2]{2.7.2}.
            
            %% 2.7.3 %%
            \item\label{subcase 2.7.3} \textbf{$\bm{j_{1} < j_{2} < i_{2}}$ (that is, $\bm{i_{1} < j_{1} < j_{2} < i_{2} < x}$).}

            This subcase follows analogously to the subcase \hyperref[subcase 1.1]{1.1}, but without any exception, and the reason for this follows from subcase \hyperref[subcase 1.3]{1.3}. The only modification is in the choice of the pair $\lno \lambda, \alpha\rno$, where the initial choice is $\lno i_1, j_2 \rno$, and if it fails, we then choose $\lno i_1, x \rno$
            
        \end{enumerate}
    \end{enumerate}
        
        %In this case, we choose $\alpha = j_2$ and $\lambda = i_1$ and claim that $F_r' \precc F_j'$.
    
        %Note that $\lno F_i' \rno^c$ can take exactly one of the following forms: $\lcu x < i_1 < i_2 \rcu$, $\lcu i_1 < x < i_2 \rcu$ or $\lcu i_1 < i_2 < x \rcu$. Similarly, $\lno F_j' \rno^c$ is exactly one of: $\lcu x < j_1 < j_2 \rcu$, $\lcu j_1 < x < j_2 \rcu$ or $\lcu j_1 < j_2 < x \rcu$. Now, if $\lno F_i' \rno^c$ is $\lcu x < i_1 < i_2 \rcu$, then $\lno F_j' \rno^c$ must be $\lcu x < j_1 < j_2 \rcu$, otherwise, $\lno F_i' \rno^c \succc \lno F_j' \rno^c$ would fail to hold, since in that case $j_1 < i_1$. Thus, $i_1 < j_1$. In fact, in all the other possible combinations of $\lno F_i' \rno^c$ and $\lno F_j' \rno^c$, we always have $i_1 < j_1$, proving our claim.
        
        %\vspace{1em}

    \caseheading{Case 3:}
        \textbf{$\bm{\left| \lno F_{i}'\rno^{c} \cap \lno F_{j}' \rno^{c} \right| = 2}$.}

        Let $\lno F_i'\rno^c = \lcu i_1,\ x,\ y \rcu$ and $\lno F_j'\rno^c = \lcu j_1,\ x,\ y \rcu$. In this case, we choose $\alpha = j_1$ and $\lambda = i_1$, that is, we choose $F_i' = F_r'$. The conclusion then follows immediately.

    This completes the proof of \Cref{thm: shellability of hexagonal grid graphs}, showing $\Delta_3 \lno H_{1 \times m \times n} \rno$ is a shellable simplicial complex.
\end{proof}

%%% Section 4: Spanning_Facets and the homotopy type of the 3-cut complex of hexagonal tiling}
\section{Spanning facets and the homotopy type of the 3-cut complex of hexagonal tiling}\label{section: 4 (Spanning facets and the homotopy type of the 3-cut complex of hexagonal tiling)}

Let $\mathcal{S}$ denote the set of spanning facets of $\Delta_3\lno H_{1 \times m \times n}\rno$ in the shelling order given in \Cref{defn: new shelling order}. Recall, the equivalent condition given in \hyperref[spanning facets condition]{$(\#)$}, states that $F_k$ is a spanning facet if for each $\lambda \in F_k$, there exists $r < k$, such that
\begin{align*}
    F_r \cap F_k = F_k \setminus \lcu \lambda \rcu.
\end{align*}

We state the main theorem of this section before building the necessary argument to prove it.

\begin{theorem}\label{thm: spanning facets and homotopy type}
    The 3-cut complex of $H_{1 \times m \times n}$ has the homotopy type of wedge of $\psi_{m,n}$ many $\lno 2m+2n+2mn-4 \rno$-dimensional spheres, where, $$\psi_{m,n} = \binom{2m+2n+2mn-1}{2} - \lsq \lno 6m+2 \rno n + (2m-4) \rsq.$$
    In other words,
    $$
    \Delta_3 \lno H_{1 \times m \times n} \rno \simeq \bigvee\limits_{\psi_{m,n}}{\mathbb{S}^{\lno 2m+2n+2mn-4 \rno}},
    $$
    where $\psi_{m,n}$ is the number of spanning facets for the shelling order.
\end{theorem}

The proof of \Cref{thm: spanning facets and homotopy type} is an immediate consequence of \Cref{lemma: explicit description of non-spanning pairs} and \Cref{lemma: all other pairs are spanning pairs}. We now proceed to prove these lemmas.

Before stating and proving these results, we must explicitly describe the elements of $S$. Unlike approaches based on Möbius inversion of the face lattice, our method explicitly identifies spanning facets, allowing the homotopy type to be read off directly from the shelling. For this, we first give the following important lemma.

\begin{lemma}\label{lemma: z is not in any spanning facet}
    For all $F \in \mathcal{S}$, the vertex $2m+2n+2mn \notin F$, that is, no spanning facet of $\Delta_3\lno H_{1 \times m \times n}\rno$ contains $2m+2n+2mn$.
\end{lemma}

\begin{proof}
    Let $\lambda = 2m+2n+2mn$ and assume that there exists an $F_j \in \mathcal{S}$, for some $1 \leq j \leq \eta_{m,n} - \beta_{m,n}$, in the \Cref{eqn: explicit shelling order} such that $\lambda \in F_j$. The choice of the upper bound $\eta_{m,n} - \beta_{m,n}$ for $j$ is explained in the remark following the proof. Using \hyperref[spanning facets condition]{$(\#)$}, there exists $k<j$ (in \Cref{eqn: explicit shelling order}) such that $$F_k \cap F_j = F_j \setminus \{\lambda\}.$$
    
     This implies that $F_k = (F_j \setminus \{\lambda\}) \cup \{\alpha\}$, for some $ \alpha \in F_j^c $, and equivalently,
    $$ 
    F_k^c = (F_j^c \setminus \{\alpha\}) \cup \{\lambda\}.
    $$
    
    Note that, $\lambda \notin F_j^c$. Let $F_j^c = \lcu j_1,\ j_2,\ j_3\rcu$ with $j_1 < j_2 < j_3$. Since $\lambda > j_3$, replacing any $\alpha \in F_j^c$ will give $F_j^c \succc F_k^c$ or equivalently $F_j \precc F_k$, which is a contradiction to $k<j$.
\end{proof}

\begin{remark}
    Using the above lemma, it is clear that none of the $T_i$ removed from the reverse lexicographic ordering and appended at the end in the order of their removal will be a spanning facet.
\end{remark}

Let $z = 2m+2n+2mn$. By the \Cref{lemma: z is not in any spanning facet}, it follows that $z$ must always lie in the complement of any spanning facet. That is, for any $F \in \mathcal{S}$, we have $$F^c = \lcu x, y, z \rcu, \quad \text{ where } x, y \in [2m+2n+2mn-1].$$

Therefore, to determine $F$ completely, we must uniquely determine the pair $(x,y)$. Note that there are exactly $$\binom{2m+2n+2mn-1}{2}$$ such pairs, which gives the binomial term appearing in $\psi_{m,n}$. However, not all of these pairs correspond to spanning facets. Some lead to what we call \textit{non-spanning pairs}.

We will explicitly describe these non-spanning pairs and prove that there are exactly $$\lno6m+2 \rno n + (2m-4)$$ of them. This gives the quantity subtracted from the binomial term in $\psi_{m,n}$.

For a better understanding of these non-spanning facets, we make some arrangements. Let $C_1,\ C_2,\dots,\ C_{2m+2n+2mn-2}$ denote the sets containing all possible pairs, that is, for $1\leq i \leq 2m+2n+2mn-2$, define
$$C_i = \lcu  \lno i,j \rno\ \middle|\ j \in \lsq i+1,\ 2m+2n+2mn-1 \rsq \rcu$$

We claim that the list of non-spanning pairs in $C_i$'s is as follows:

\begin{enumerate}
    \item For $1\leq i \leq m-1$, the non-spanning pairs are $\lno i,i+1 \rno$,  $\lno i,i+m \rno$, $\lno i,i+m+1 \rno$ and $\lno i,i+m+n+mn \rno$. So, there are a total of $4(m-1)$ such non-spanning pairs.

    \item For $i=m$, the non-spanning pairs are $\lno i,i+m \rno$, $\lno i,i+m+1 \rno$ and $\lno i,i+m+n+mn \rno$. So, there are a total of 3 such non-spanning pairs.
    
    \item Fix $k_1 \in [n-1]$. For $k_1 m + k_1 \leq i \leq k_1 m + k_1 + m - 1$, the non-spanning pairs are $\lno i,i+1 \rno$, $\lno i,i+m+1 \rno$, $\lno i,i+m+2 \rno$, $\lno i,i+(m+1)n \rno$ and 
    
    $\lno i,i+(m+n+mn)+1 \rno$. So, there are a total of $5m(n-1)$ such non-spanning pairs. 

    \item For $i = k_2(m+1) - 1$, where $k_2 \in [2,n]$, the non-spanning pairs are $\lno i,i+(m+1) \rno$ and $\lno i,i+(m+1)n \rno$. So, there are a total of $2(n-1)$ such non-spanning pairs.

    \item For $n+mn+1\leq i \leq m+n+mn-1$, the non-spanning pairs are $\lno i,i+1 \rno$, $\lno i,i+(m+1)n \rno$, and $\lno i,i+m+n+mn \rno$. So, there are a total of $3(m-1)$ such non-spanning pairs.

    \item For $i = (m+1)n$, the non-spanning pairs are $\lno i,i+1 \rno$, and $\lno i,i+(m+1)n \rno$. So, there are a total of $2$ such non-spanning pairs.

    \item For $i=m+n+mn$, the non-spanning pair is $\lno i,i+(m+1)n \rno$. So, there is 1 such non-spanning pair.

    \item For all 
    \begin{align*}
        i \in &[m+n+mn+1,\ 2m+2n+2mn-m-1] \\
    &\setminus \lno \lcu 2n+2mn \rcu \sqcup \lcu (m+n+mn) + t(m+1)\ \middle|\ t\in \lsq n-1 \rsq \rcu \rno,
    \end{align*} the non-spanning pair is $\lno i,i+m+1 \rno$. So, there are a total of $$(m+n+mn-m-1)-(n) = mn-1$$ such non-spanning pairs.
\end{enumerate}

For all other $i$'s which are not listed above, have no non-spanning pairs, that is, if $$i=\lcu 2n+2mn \rcu,$$ or, 
\begin{align*}
    i \in & \lcu (m+n+mn) + t(m+1)\ \middle|\ t\in \lcu 1,\ 2,\dots,\ n \rcu \rcu\\
    & \cup \lsq m+2n+2mn+1,\ 2m+2n+2mn-2\rsq,
\end{align*} then $C_i$ do not contain any non-spanning pair. The next lemma proves our claim that the above-listed pairs are indeed non-spanning pairs.

\begin{lemma}\label{lemma: explicit description of non-spanning pairs}
    The pairs listed in $C_i$, for $1 \leq i \leq 2m + 2n + 2mn - 2$ are non-spanning pairs, with a total $6mn+2m+2n-6$ of them.
\end{lemma}

\begin{proof}
    Recall that, using the condition \hyperref[spanning facets condition]{$(\#)$}, we know that if $F$ is a spanning facet for a given shelling order, then for each $\lambda \in F$, there exists an $k$ such that $F_k \precc F$ and $$F_k\ \cap\ F = F \setminus \lcu \lambda \rcu.$$

    In other words, $$F_k = \lno F \cup \lcu \alpha \rcu \rno \setminus \lcu \lambda \rcu,$$ where, $\alpha \in F^c$. 

    Recall that $z = 2m+2n+2mn$. If the complement of any of the above pairs, along with the $z$, does not form a spanning facet, then the condition in \hyperref[spanning facets condition]{$(\#)$} must fail. Thus, to prove that the above pairs are non-spanning, it suffices to find a vertex in the facets formed by a pair defined above that fails \hyperref[spanning facets condition]{$(\#)$}. Let this vertex be $\lambda_{x,y}$ for the pair $\lno x, y \rno$. Also, let $S_{x,y}$ denote the facets formed by the pair $\lno x, y \rno$ (that is, $S_{x,y}^c = \lcu x, y, z \rcu$), with $j_{x,y}$ being the position of $S_{x,y}$ in the \Cref{eqn: explicit shelling order}. In other words, $S_{x,y} = F_{j_{x,y}}$, where $1 \leq j_{x,y} \leq \eta_{m,n} - \beta_{m,n}$.
    
    We group the non-spanning pairs defined above into the following types and, for each case, explicitly define a vertex $\lambda_{x,y}$ for which the spanning condition fails for $S_{x,y}$.
    
    \begin{enumerate}[leftmargin = *, label = \textbf{Type} \textbf{\arabic*}.]
        %%% Type 1 %%%
        \item\label{Type 1} \textbf{The pairs of the form $(i,i+1)$.}
        
        We define $\lambda_{i,i+1}$ for these pairs as follows, based on the indices where $i$ varies,

        \begin{enumerate}
            \item For $1 \leq i \leq m-1$ or $\lno m+1 \rno n \leq i \leq m+n+mn-1$, we have, $$\lambda_{i,i+1} = i+m+n+mn+1.$$

            \item For some fixed $k_1 \in \lsq n-1 \rsq$, and $k_1m + k_1 \leq i \leq k_1m+k_1+m-1$, we have
            $$\lambda_{i,i+1} = i+m+n+mn+2.$$
        \end{enumerate}
        
        %%% Type 2 %%%
        \item\label{Type 2} \textbf{The pairs of the form $(i,i+m)$.}
        
        For $1 \leq i \leq m$, we define $\lambda_{i,i+m}$ for these pairs as, $$\lambda_{i,i+m} = \lno m + 1 \rno n + i.$$

        %%% Type 3 %%%
        \item\label{Type 3} \textbf{The pairs of the form $(i,i+m+1)$.}
        
        We define $\lambda_{i,i+m+1}$ for these pairs as follows, based on the indices where $i$ varies,

        \begin{enumerate}
            \item\label{Type 3(a)} For $1 \leq i \leq m$ or $i = k_2 \lno m+1 \rno -1$, where $k_2 \in \lsq 2,\ n \rsq$, we have, $$\lambda_{i,i+m+1} = i+m+n+mn+1.$$

            \item\label{Type 3(b)} For some fixed $k_1 \in \lsq n-1 \rsq$, and $k_1m + k_1 \leq i \leq k_1m+k_1+m-1$, we have
            $$\lambda_{i,i+m+1} = i+m+n+mn+2.$$

            \item\label{Type 3(c)} For all, 
            \begin{align*}
                i \in &[m+n+mn+1,\ 2m+2n+2mn-m-1] \\
                &\setminus \lno \lcu 2n+2mn \rcu \sqcup \lcu (m+n+mn) + t(m+1)\ \middle|\ t\in \lcu 1,\ 2,\dots,\ n \rcu \rcu \rno,
                \end{align*} we have,
                $$\lambda_{i,i+m+1} = i+m+2.$$
        \end{enumerate}

        %%% Type 4 %%%
        \item \textbf{The pairs of the form $(i,i+m+2)$.}
        
        Fix $k_1 \in \lsq n-1 \rsq$. For $k_1m + k_1 \leq i \leq k_1m+k_1+m-1$, we define $\lambda_{i,i+m+2} $ for these pairs as, $$\lambda_{i,i+m+2} = i+m+n+mn+2.$$

        %%% Type 5 %%%
        \item \textbf{The pairs of the form $\lno i,i+\lno m+1 \rno n \rno$.}
        
        We define $\lambda_{i,i+(m+1)n}$ for these pairs as follows, based on the indices where $i$ varies,

        \begin{enumerate}
            \item Fix $k_1 \in \lsq n-1 \rsq$. For $k_1m + k_1 \leq i \leq k_1m+k_1+m-1$, we have the following two choices for $\lambda_{i,i+(m+1)n}$,
            $$\lambda_{i,i+(m+1)n} = i+m+n+mn+1 \quad \text{or} \quad \lambda_i = i+m+n+mn+2.$$
            
            \item For $i = k_2 \lno m+1 \rno -1$, where $k_2 \in \lsq 2,\ n \rsq$, or $i=\lno m+1 \rno$, we have, $$\lambda_{i,i+(m+1)n} = i+m+n+mn+1.$$

            \item For $n + mn + 1 \leq i \leq m+n+mn-1$, we have,
            $$\lambda_{i,i+(m+1)n} = i+m+n+mn \quad \text{or} \quad \lambda_{i,i+(m+1)n} = i+m+n+mn+1.$$

            \item For $i=m+n+mn$, we define $\lambda_{i,i+(m+1)n}$ as follows,
            $$\lambda_{i,i+(m+1)n} = i+m+n+mn.$$
        \end{enumerate}

        %%% Type 6 %%%
        \item \textbf{The pairs of the form $(i,i+m+n+mn)$.}
        
        For $1 \leq i \leq m$, or $n+mn+1 \leq i \leq m+n+mn-1$, we define $\lambda_{i,i+m+n+mn}$ for these pairs as, $$\lambda_{i,i+m+n+mn} = i+m+n+mn+1.$$

        %%% Type 7 %%%
        \item \textbf{The pairs of the form $(i,i+m+n+mn+1)$.}
        
        Fix $k_1 \in \lsq n-1 \rsq$. For $k_1m + k_1 \leq i \leq k_1m+k_1+m-1$, we define $\lambda_{i,i+m+n+mn+1}$ for these pairs as,
        $$\lambda_{i,i+m+n+mn+1} = i+m+n+mn+2.$$
    \end{enumerate}

    We now analyze the possible $\alpha$ for all the pairs listed above. For all the $\lambda_{x,y}$ defined above for the pairs $\lno x, y \rno$, observe that $\lambda_{x,y} > x$ and $\lambda_{x,y} > y$. So, if we choose $\alpha$ to be either $x$ or $y$, then $S_{x,y} \precc F_k$, contradicting $k < j_{x,y}$.

    Now, if $\alpha = z$, then we first deal with \hyperref[Type 3(c)]{Type 3(c)} and then the rest. For \hyperref[Type 3(c)]{Type 3(c)}, $$F_k^c = \lcu i,\ i+m+1,\ \lambda_{x,y} \rcu = \lcu i,\ i+m+1,\ i+m+2 \rcu,$$ where the range of $i$ is specified in \hyperref[Type 3(c)]{Type 3(c)}. In this case, $F_k = \mathcal{N}\lno i - (m+1)n\rno$, which is one of the $T_i$, implying $S_{x,y} \precc F_k$ since $T_i$ are placed at the end. For all other types, $S_{x,y}^c$ is connected (that is, it induces a path $P_3$), contradicting $F_k^c$ being disconnected.

    Hence, all the pairs defined above are non-spanning pairs. Moreover, summing up all of these non-spanning pairs gives a total of $6mn+2m+2n-6$.
\end{proof}

Observe that the pairs that were not listed above as non-spanning are indeed spanning. This follows from two facts. First, the analysis shown in the previous proof explicitly lists out all those pairs for which the resulting facet $S_{x,y}$ will fail \hyperref[spanning facets condition]{$(\#)$} because 
\begin{enumerate}[label = \roman*.]
\itemsep0em
    \item any constructed $F_k$ is forced to appear after $S_{x,y}$,
    \item any constructed $F_k^c$ is connected (induces a $P_3$), or
    \item any constructed $F_k$ coincides with one of the removed facets $T_i$.
\end{enumerate} 

Second, for any remaining pair $\{x,y\}$, one can construct a valid $F_k$ by choosing a suitable $\alpha \in S_{x,y}^c$ that will avoid all three obstructions listed above for some $\lambda \in S_{x,y}$. The next result proves that the pairs given in types 1 to 7 in the proof of \Cref{lemma: explicit description of non-spanning pairs} are the only non-spanning pairs.

\begin{lemma}\label{lemma: all other pairs are spanning pairs}
    All pairs that are not listed in $C_i$, for $1 \leq i \leq 2m + 2n + 2mn -2$ are spanning pairs.
\end{lemma}

\begin{proof}
    Let $\lno a,b \rno$ be a pair that does not fit into any of the types $1$ through $7$ described in the proof of \Cref{lemma: explicit description of non-spanning pairs}. To show that $S_{a,b}$ is a spanning facet, it suffices to verify that for each $\lambda \in S_{a,b}$ there exists a facet $F_r$ with $r < j_{a,b}$ (equivalently, $F_r \precc S_{a,b}$) satisfying \hyperref[spanning facets condition]{$(\#)$}.

    Observe that if the pair $\lno a,b \rno$ does not belong to any of the types $1$ through $7$, then exactly one of the following three cases must occur.

    \caseheading{Case 1:} The vertices $a$ and $b$ are adjacent, that is, $\lcu a, b \rcu \in E\lno H_{1 \times m \times n} \rno$.

    It follows from the description of non-spanning pairs given in types $1$ through $7$ that this case can occur only when the pair $\lno a, b \rno$ admits one of the following forms:

    \begin{enumerate}[label = \roman*.]
        \item $\lno a, b \rno = \lno i, i + m + n + mn + 1 \rno$, when $i \in [m] \sqcup \lsq n + mn,\ m + n + mn - 2\rsq$; or
        
        \item $\lno a, b \rno = \lno i, i + m + n + mn + 2 \rno$, when $i \in \lsq k_1 m + k_1,\ \lno k_1 + 1 \rno m + k_1 - 1 \rsq$, for all $k_1 \in [n-1]$.
    \end{enumerate}

    Note that, for any $\lambda \in S_{a,b}$ we choose to swap with $\alpha \in S_{a, b}^c$, the resulting $F_r^c$ will never coincide with any $T_i^c$, for $1 \leq i \leq \beta_{m,n}$. This is because $a$ and $b$ are adjacent to each other, whereas by \Cref{obs: Ti^c is an independent set}, each $T_i^c$ is an independent set.

    Let $\lambda \in S_{a,b}$. First, suppose that $\lambda$ is adjacent to neither $a$ nor $b$. In this subcase, we choose $\alpha = z$, which yields $$ F_r^c = \lcu a, b, \lambda \rcu.$$ Since every element of $S_{a,b}$ is smaller than $z$, we have $F_r^c \succc S_{a,b}^c$. Moreover, the $F_r^c$ obtained here does not induce a $P_3$, as $\lambda$ is adjacent to neither $a$ nor $b$.

    Now suppose that $\lambda$ is adjacent to either $a$ or $b$. In this case, choosing $\alpha = z$ is not possible, since the resulting $F_r^c$ would induce a $P_3$. Before choosing $\alpha$, we observe that in all the possible forms of $\lno a, b \rno$ listed above, we have $a \in V_1$ and $b \in V_2$, and hence $a < b$. In fact, $b$ is the largest element in the neighborhood of $a$, that is, $b = \max \{N\lno a \rno\}$.

    For this subcase, we choose $\alpha = b$. It suffices to show that $\lambda < b$, which guarantees $F_r^c \succc S_{a,b}^c$. Indeed, if $\lambda$ is adjacent to $a$, then by the above observation $\lambda < b$; and if $\lambda$ is adjacent to $b$, then by \Cref{obs: V1 nbhd c}, we have $\lambda \in V_1$, again implying $\lambda < b$. Finally, the resulting $F_r^c$ does not induce a $P_3$, since in all the cases listed above the vertex $b$ is at a distance of at least $2$ from $z$.
    
    \caseheading{Case 2:} The vertices $a$ and $b$ are at a distance of two from each other, that is, $a$ and $b$ are endpoints of an induced $P_3$.

    This case is only possible when $\lno a, b\rno$ admits the following forms:

    \begin{enumerate}[label =\roman*.]
        \item $\lno a, b \rno = \lno i, i+1 \rno$, when $$i \in \lsq \lno k_2 - 1 \rno \lno m + 1 \rno + m + n + mn + 1, k_2 \lno m + 1 \rno + m + n + mn - 1 \rsq,$$
        for each $k_2 \in \lsq n \rsq$, or when $$i \in \lsq m + 2n + 2mn,\ 2m + 2n + 2mn - 2 \rsq.$$%The only exception is that $\lno a, b \rno$ does not assume the form $\lno i, i+m+2 \rno$ when $i = n(m+1)+m+n+mn-1$, since in that case $$i+m+2 = z = 2m+2n+2mn.$$

        \item $\lno a, b \rno = \lno i, i+m+1 \rno$, when $$i = m + n + mn + k_3 \lno m+1 \rno,$$ for all $k_3 \in \lsq n-1 \rsq$, or when, $$i \in \lsq 2n + 2mn, m + 2n + 2mn - 2 \rsq.$$

        \item $\lno a, b \rno = \lno i, i+m+2 \rno$, when $$i \in \lsq \lno k_3 - 1 \rno \lno m + 1 \rno + m + n + mn + 1, k_3 \lno m + 1 \rno + m + n + mn - 1 \rsq,$$
        for each $k_3 \in \lsq n-1\rsq$.
    \end{enumerate}

    Before proceeding, we record some observations. In all the above cases, both $a$ and $b$ belong to $V_2$. Hence, by \Cref{obs: V1 nbhd c}, every vertex adjacent to either $a$ or $b$ must lie in $V_1$.

    Let $\lambda \in S_{a,b}$. We consider the following two subcases.

    \begin{enumerate}
        \item $\lambda$ is adjacent to both $a$ and $b$.
    
        In this situation, choosing $\alpha = z$ is not possible, since the resulting $F_r^c$ will induce a $P_3$. Instead, we choose $\alpha = b$. Since $\lambda$ is adjacent to $b$, we have $\lambda < b$, and hence $F_r^c \succc S_{a,b}^c$.
    
        Moreover, $F_r^c$ cannot coincide with any $T_i^c$, as $\lambda$ is adjacent to $a$ and each $T_i^c$ is an independent set by \Cref{obs: Ti^c is an independent set}. Finally, $F_r^c$ does not induce $P_3$, since $a$ and $z$ are always at a distance greater than two. The only exception when $F_r^c$ induces $P_3$ occurs when $\lambda = m+n+mn-1$ and $$\lno a, b \rno = \lno m+2n+2mn-1,\; m+2n+2mn \rno.$$
        This exceptional case can be handled by choosing $\alpha = a$ instead of $b$, while keeping all other arguments unchanged.
    
        \item $\lambda$ is at a distance of at least two from at least one of $a$ or $b$.
    
        First, suppose that $\lcu a, b, \lambda \rcu$ coincides with some $T_i^c$. In this case, choosing $\alpha = z$ is not possible, as it would yield $F_r^c = T_i^c$. We therefore choose $\alpha = b$. As in the previous subcase, this choice guarantees $F_r^c \succc S_{a,b}^c$.
    
        Moreover, the resulting $F_r^c$ cannot coincide with any $T_i^c$, since it contains $z$ and no $T_i^c$ contains $z$. Also, $F_r^c$ does not induce a $P_3$, as $a$ is adjacent to neither $\lambda$ nor $z$.
    
        If the above situation does not arise, then we may safely choose $\alpha = z$. This again yields $F_r^c \succc S_{a,b}^c$, and the resulting $F_r^c$ neither coincides with any $T_i^c$ nor induces a $P_3$, by arguments analogous to those above.
    \end{enumerate}

    \caseheading{Case 3:} The vertices $a$ and $b$ are at distance greater than two from each other.

    We do not list explicit types in this case, since it consists of all pairs $\lno a,b \rno$ that were not covered earlier. Here, we choose $\alpha = z$. Then the resulting facet $F_r^c$ satisfies $F_r^c \succc S_{a,b}^c$, as $\lambda < z$ for every $\lambda \in S_{a,b}$. 

    As in the previous cases, the facet $F_r^c$ obtained here neither coincides with any $T_i^c$ nor induces a $P_3$. Both assertions follow from arguments analogous to those used above.

    This completes the proof of \Cref{lemma: all other pairs are spanning pairs}
\end{proof}

Using \Cref{lemma: explicit description of non-spanning pairs} and \Cref{lemma: all other pairs are spanning pairs}, the proof of \Cref{thm: spanning facets and homotopy type} follows immediately.

%%% Section 5: Future directions %%%
\section{Future directions}

Motivated by our results on the shellability of the $3$-cut complex of hexagonal grid graphs, we investigated whether similar results hold for higher values of $k$. Computational experiments using \texttt{SageMath} suggest that similar behavior persists for $k \geq 3$, particularly for hexagonal line tilings. This leads to the following question.

\begin{ques}
    Is the $k$-cut complex of the hexagonal line tiling shellable for all $k \geq 3$?
\end{ques}

A natural extension of this question is to ask whether a similar shellability result holds for the $k$-cut complex of the general hexagonal grid graphs $H_{r \times s \times t}$.

We now reframe this question from a structural perspective. Observe that hexagonal grid graphs are bipartite, as can be seen from the vertex labeling given earlier for $H_{1 \times m \times n}$. Moreover, Bayer et al. proved shellability results for $k$-cut complexes of various square grid graphs, which are also bipartite.

Motivated by these observations, we ask the following.

\begin{ques}
    Is the $k$-cut complex of a bipartite graph shellable?
\end{ques}

This question relates to earlier work of Bayer et al., who established shellability for complete bipartite and complete multipartite graphs. Extending these results to general bipartite graphs would lead to a better understanding of how bipartite structure influences the shellability of cut complexes.

\section*{Acknowledgements}
The author gratefully acknowledges the assistance provided by the Council of Scientific and Industrial Research (CSIR), India, through grant 09/1237(15675)/2022-EMR-I. The author also sincerely thanks Dr. Anurag Singh, the author’s supervisor, for his valuable feedback and continued support.

%%% REFERENCES %%%
\nocite{}
\printbibliography 
    
\end{document}